\documentclass{article}
\usepackage[utf8]{inputenc}
\usepackage[margin=1in]{geometry}
\usepackage{amsmath} % assumes amsmath package installed
\usepackage{todonotes}
\usepackage[nolist]{acronym}
\usepackage{xcolor}
\usepackage{amsmath}
\usepackage{amssymb}
\usepackage{mathtools}
\usepackage{subcaption}
\usepackage{xcolor}
\usepackage{hyperref}
\hypersetup{colorlinks=true, linkcolor=black}
\usepackage{algorithm}
\usepackage{algpseudocode} 
\usepackage{mathrsfs}
\usepackage{cite}
\usepackage{adjustbox}
\usepackage{float}

\newcommand{\R}{\mathbb{R}} 
\newcommand{\K}{\mathbb{K}} 
 
\newcommand{\N}{\mathbb{N}}

\newcommand{\Z}{\mathbb{Z}}

\newcommand{\A}{\mathcal{A}}
\newcommand{\0}{\mathbf{0}}

\newcommand{\psd}{\mathbb{S}}

\DeclarePairedDelimiter{\abs}{\lvert}{\rvert}
\DeclarePairedDelimiter{\norm}{\lVert}{\rVert}

\DeclarePairedDelimiterX{\inp}[2]{\langle}{\rangle}{#1, #2}

\DeclarePairedDelimiter{\Mp}{\mathcal{M}_+(}{)}

\usepackage{amsthm}
\newtheorem{thm}{Theorem}[section]
\newtheorem{lem}[thm]{Lemma}
\newtheorem{prop}[thm]{Proposition}
\newtheorem{cor}{Corollary}
\newtheorem{defn}{Definition}[section]

\newtheorem{exmp}{Example}[section]
\newtheorem{rem}{Remark} 
 
\newtheorem{prob}{Problem} 
\DeclarePairedDelimiterX{\vol}[2]{\textrm{Vol}_{#1}(#2}{)}{#1, #2}

\newcommand{\bbmu}{\boldsymbol{\mu}}
\newcommand{\bell}{\boldsymbol{\ell}}

\renewcommand\footnotemark{}

\title{\textbf{Maximizing Slice-Volumes of Semialgebraic Sets \\ using Sum-of-Squares Programming}
}

\author{Jared Miller$^1$, Chiara Meroni$^2$, Matteo Tacchi$^3$,  Mauricio Velasco$^4$
\thanks{$^1$J. Miller is with the Automatic Control Laboratory (IfA), Department of Information Technology and Electrical Engineering (D-ITET), ETH Z\"{u}rich, Physikstrasse 3, 8092, Z\"{u}rich, Switzerland (\href{mailto:jarmiller@control.ee.ethz.ch}{jarmiller@control.ee.ethz.ch}).}
\thanks{$^2$ C. Meroni is with the ETH Institute for Theoretical Studies, Z\"urich, Switzerland (\href{mailto:chiara.meroni@eth-its.ethz.ch}{chiara.meroni@eth-its.ethz.ch}).}
\thanks{$^3$ M. Tacchi is with the Univ. Grenoble Alpes, CNRS, Grenoble INP (Institute of Engineering Univ. Grenoble Alpes), GIPSA-lab, 38000 Grenoble, France (\href{mailto:matteo.tacchi@gipsa-lab.fr}{matteo.tacchi@gipsa-lab.fr}).}
\thanks{$^4$ M. Velasco is with the Departamento de Inform\'{a}tica, Universidad Cat\'{o}lica del Uruguay, Av. 8 de Octubre 2738, 11600 Montevideo, Departamento de Montevideo, Uruguay (\href{mailto:mauricio.velasco@ucu.edu.uy}{mauricio.velasco@ucu.edu.uy}).
}
\thanks{J. Miller was partially supported by NSF grants  CNS--1646121, ECCS--1808381 and CNS--2038493, AFOSR grant FA9550-19-1-0005, ONR grant N00014-21-1-2431, and the Swiss National Science Foundation under NCCR Automation, grant agreement 51NF40\_180545.}
\thanks{Mauricio Velasco was partially supported by ANII (Uruguay) via Fondo Clemente Estable grant FCE-1-2023-1-176172.}
}

\begin{document}

\maketitle
% \thispagestyle{empty}
% \pagestyle{empty}

%%%%%%%%%%%%%%%%%%%%%%%%%%%%%%%%%%%%%%%%%%%%%%%%%%%%%%%%%%%%%%

\begin{abstract}
\label{sec:abstract}

% \urg{Abstract goes here}

This paper presents an algorithm to maximize the volume of an affine slice through a given semi-algebraic set. This slice-volume task is formulated as an infinite-dimensional linear program in continuous functions, inspired by prior work in volume computation of semialgebraic sets. A convergent sequence of upper-bounds to the maximal slice volume are computed using the moment-Sum-of-Squares hierarchy of semidefinite programs in increasing size. The computational complexity of this scheme can be reduced by utilizing topological structure (in dimensions 2, 3, 4, 8) and symmetry.
This numerical convergence can be accelerated through the introduction of redundant Stokes-based constraints. Demonstrations of slice-volume calculation are performed on example sets.

\end{abstract}


\section{Introduction}
\label{sec:introduction}

% This paper uses polynomial optimization methods in order to upper-bound the maximal slice-volume of a set.
Given a compact set $L \in \R^{n}$ contained within a ball of finite radius $R>0$, the slice-volume-maximizing problem considered in this paper is:
\begin{prob}
\label{prob:slice}
Find a direction $\theta$ and an affine offset $t$ to supremize \begin{align}
    P^* =& \sup_{\theta, t} \textrm{Vol}_{n-1}\left((\theta \cdot x = t) \cap L \right) \label{eq:slice}\\
    & \theta \in S^{n-1}, \ t \in [-R, R]. \nonumber 
\end{align}
\end{prob}

The slice-volume task in Problem \ref{prob:slice} is a particular form of a parametrically-defined volume approximation problem. The slice-volume task can also be interpreted as finding the supremal value of the Radon transform of the $0/1$ indicator function on $L$ \cite{helgason1980radon}.

% \urg{TODO: Discuss the complexity of volume approximation. NP hard in general, even for polytopes. Triangulation-based schemes for computing volumes. Raid the reference list of} \cite{henrion2009approximate},  \cite{lairez2019computing}, \cite{tacchi2023stokes}.

% \urg{Slice-volume background. Intersection bodies, Busemann-Petty problem. Fourier Slice Theorem?}
Slices of sets and computation of slice-volumes are of interest in different areas of mathematics. 
% Slice-volume computation and re
One such research area is geometric tomography, which involves the reconstruction of shapes from an assemblage of their slices along multiple $(\theta, t)$ cuts \cite{Gardnerbook}. The development of geometric tomomography (for convex bodies) was motivated by the Busemann-Petty problem, which asked (for convex and usually symmetric shapes $L, T$) if the relation $\forall \theta \in S^n: \textrm{Vol}_{n-1}\left((\theta \cdot x =0 ) \cap L \right) \geq \textrm{Vol}_{n-1}\left((\theta \cdot x = 0) \cap T \right)$ implied that $\textrm{Vol}_n(L) \geq \textrm{Vol}_n(T)$. This implication is true in dimensions $n \leq 4$, and has counterexamples with $n \geq 5$ \cite{Gardner19941,Gardner19942,Koldobsky1998,GKS1999,Zhang1999}.
% From the point of view of convex geometry, this area of research is sometimes referred to as geometric tomography \cite[Chapter 8]{Gardnerbook}, and it developed around and beyond the famous Busemann-Petty problem \cite{Gardner19941,Gardner19942,Koldobsky1998,GKS1999,Zhang1999}. 
The study of the Busemann-Petty problem motivated the creation of new geometric objects (e.g.,  intersection bodies), and the establishment of currently open problems such as Bourgain's slicing problem and the KLS conjecture 
% This led to the introduction of new objects (e.g., intersection bodies), and open conjectures such as Bourgain's slicing problem and the KLS conjecture, which are very active areas of research 
\cite{KlaMil22:SlicingProblem,klartagLehec2022bourgains,Klartag:LogBoundBourgain,GiannopoulosKoldobskyZvavitch2023}.
A full characterization of slice-properties (e.g., combinatorial types, volume) has been conducted for polytopes
% From a more discrete point of view, there has been interest in extremal slices of polytopes, 
such as simplices \cite{webb,Koenig2021,Walkup68:Simplex5D}, cubes \cite{Kball1,Kball2,Lawrence79,pournin-scubesections,Koenig2021,fukudaMutoGoto97-5cubeslices,Moodyetal2013}, cross-polytopes \cite{Liu+Tkocz,Koenig2021}, and other norm-balls \cite{MeyerPajor88,Khovanskii-slices2006}. The combinatorial structure of central hyperplane sections $(t=0)$ for polytopes can be studied with purely algebraic tools for all polytopes \cite{BBMS:IntersectionBodiesPolytopes}.
% In this context the rich combinatorics of central hyperplane sections has been investigated: these can be studied with purely algebraic tools for all polytopes \cite{BBMS:IntersectionBodiesPolytopes}. More recently, there has been interest in studying all affine hyperplane sections.
The work in \cite{brandenburg2023slice} investigated affine hyperplane sections $(t \in [-R, R])$ for polytopes, with a full classification of combinatorial types of sliced sets based on a hyperplane arrangement. The research in \cite{brandenburg2023slice} exactly solved of the polyhedral slice-volume  in polynomial time (in fixed dimension).

The dominant class of algorithms to compute volumes of nonconvex bodies are Monte-Carlo methods \cite{lovasz2006simulated, isaac2023algorithm}, considering that even volume computation of polytopes is \#-P hard in general \cite{dyer1988complexity}.
The moment-\ac{SOS} hierarchy of \acp{SDP} \cite{parrilo2000structured, lasserre2009moments} offers one approach to compute the volume of \ac{BSA} sets \cite{henrion2009approximate, tacchi2022exploiting}. The volume computation problem is posed as a primal-dual pair of infinite-dimensional \acp{LP} over nonnegative Borel measures and continuous functions, in which the Lebesgue measure on $L$ is a feasible and optimal solution. The function formulation can be interpreted as finding a smooth over-approximator to the indicator function with minimum Lebesgue integral.
The primal-dual \acp{SDP} in the hierarchy correspond to increasing the polynomial degree of the indicator-approximator (or the number of moments considered on the dual), resulting in a nonincreasing sequence of upper bounds to the volume of the true set. 

% This sequence of bounds  will converge in increasing degree $k$ at a rate of

This sequence of bounds will converge in degree $k$ as $O(k^{-z})$ for some constant $z>0$ \cite{schlosser2024convergence} (improved from a  $O(1/\log \log k)$ bound in \cite{korda2018convergence}). Such a slow convergence rate is partly due to Gibbs phenomena (nonvanishing oscillations) in the one-sided approximation of the discontinuous indicator function by smooth polynomials. Redundant Stokes-based constraints on the boundary-defining polynomials do not change the \ac{LP} optima, but allow for additional degrees of freedom in the finite-degree truncation and empirically demonstrate improved convergence \cite{lasserre2017computing, tacchi2023stokes}. Stokes relations were also used to compute moments of the Hausdorff measure of \ac{BSA} sets using the moment-\ac{SOS} hierarchy \cite{lasserre2020boundary}.

Moment-\ac{SOS}-based volume approximation also has applications in the analysis and control of dynamical systems. Problem instances include reachable set and region of attraction estimation \cite{Henrion_2014, majumdar2014convex} (from outside) and \cite{korda2013inner} (from inside), maximum positively invariant sets \cite{oustry2019inner}, maximum controlled invariant sets \cite{korda2014convex}, and global attractors \cite{schlosser2021converging}. It remains an open problem to generate and apply Stokes constraints towards dynamical systems volume maximization programs in order to sharpen convergence.

We also note that \cite{lairez2019computing} computes the volume of semialgebraic sets through Picard-Fuchs formulae for the period of rational integrals. They recursively perform volume computation by aggregating the volume of lower-dimensional slices using Oaku's method for parameter-dependent integration over semialgebraic sets \cite{oaku2013algorithms}, but they do not optimize to find the maximal slice-volume.

% Sparsity in the constraint-defining polynomials may also be exploited to reduce computational complexity of vol \cite{tacchi2022exploiting}.

% Their work uses Oaku's methods for the 

% lasserre2020computing

% real algebraic approaches to perform volume computation include symbolic methods 

% Increasing the polynomial degree results in a 

% One method to compute the volume of \ac{BSA} sets is 

% \urg{Introduction goes here. }

% \urg{Literature review here.}

The contributions of this work are:
\begin{itemize}
    \item A pair of infinite-dimensional \acp{LP} to solve the slice-volume Problem \ref{prob:slice}.
    \item An accounting of computational complexity for moment-\ac{SOS} truncations of these \acp{LP}.
    \item Identification of complexity reduction mechanisms using algebraic and topological structure.
    \item Application of Stokes-based methods to improve the numerical convergence of slice-volume approximation, including the introduction of Stokes schemes that respect symmetries.    
\end{itemize}

To the best of our knowledge, this is the first work that performs slice-volume maximization of semialgebraic sets (beyond simpler convex sets such as polytopes \cite{brandenburg2023slice} and ellipsoids).

This paper has the following structure: Section \ref{sec:preliminaries} reviews preliminaries including notation, measure theory, and existing moment-\ac{SOS} methods for computing the (standard) volume of sets. Section \ref{sec:slice_volume} poses a primal-dual pair of infinite-dimensional convex \acp{LP} that solve the slice-volume task in Problem \ref{prob:slice}. Section \ref{sec:slice_sos} performs and analyzes a moment-\ac{SOS} truncation of the slice-volume \acp{LP} into a hierarchy of finite-dimensional \acp{SDP} in increasing size. Section \ref{sec:reduce_complex} reduces the computational complexity of these \ac{SOS} \acp{SDP} by applying symmetry, algebraic structure, and topological structure (in dimensions $2, 3, 4,$ and $8$). Section \ref{sec:stokes} utilizes Stokes methods from \cite{tacchi2023stokes} to further improve the numerical performance of slice-volume \ac{SOS} programs. Section \ref{sec:examples} demonstrates this slice-volume approximation scheme on numerical examples. Section \ref{sec:extensions} extends the slice-volume maximization work to maximize and approximate the Radon transform of functions.
Section \ref{sec:conclusion} contains conclusions the paper. Appendix \ref{app:duality} provides a proof of strong duality between the slice-volume infinite-dimensional \ac{LP} in measures and in continuous functions. Appendix \ref{app:poly_approx} proves that polynomial auxiliary functions can be used to solve the slice-volume task. Appendix \ref{app:duality_stokes} proves strong duality for the Stokes-constrained slice-volume program.
% \urg{Fill in the paper structure}
% Section \ref{sec:preliminaries} will review preliminaries such as notation, notions of stability for linear systems, and \ac{SOS} proofs of polynomial nonnegativity. Section \ref{sec:full_method} will present 
% The paper is concluded in Section \ref{sec:conclusion}.
\section{Preliminaries}
\label{sec:preliminaries}

% \subsection{Acronyms/Initialisms}
\begin{acronym}[WSOS]
\acro{BSA}{Basic Semialgebraic}
\acroindefinite{BSA}{a}{a}
% \acro{CSP}{Correlative Sparsity Pattern}

\acro{LP}{Linear Program}
% \acroplural{LMI}[LMIs]{Linear Matrix Inequalities}
\acroindefinite{LP}{an}{a}

% \acro{LMI}{Linear Matrix Inequality}
% \acroplural{LMI}[LMIs]{Linear Matrix Inequalities}
% \acroindefinite{LMI}{an}{a}

% \acro{LQR}{Linear Quadratic Regulator}
% \acroplural{LMI}[LMIs]{Linear Matrix Inequalities}
% \acroindefinite{LQR}{an}{a}

% \acro{LP}{Linear Program}
% \acroindefinite{LP}{an}{a}
% \acro{OCP}{Optimal Control Problem}

% \acro{ODE}{Ordinary Differential Equation}

% \acro{POP}{Polynomial Optimization Problem}

\acro{PSD}{Positive Semidefinite}

% \acro{PD}{Positive Definite}

% \acro{PDE}{Partial Differential Equation}

\acro{SDP}{Semidefinite Program}
\acroindefinite{SDP}{an}{a}

\acro{SOS}{Sum of Squares}
\acroindefinite{SOS}{an}{a}

\acro{WSOS}{Weighted Sum of Squares}

\end{acronym}

\subsection{Notation}
% \urg{Fill in the notation}
The $n$-dimensional real Euclidean vector space is $\R^n$. The dot product between vectors $x, y \in \R^n$ is $x \cdot y = \sum_i x_i y_i$. The $n$-dimensional unit ball of radius $R$ is $B_R^n = \{x \in \R^n \mid \norm{x}_2 \leq R\}.$ The set of natural numbers is $\N$, and the subset of natural numbers between $a$ and $b$ is $a..b$. Given a set of indices $a..b$ and a set of elements $\{g_i\}_{i=a}^b,$ an index $i \in a..b$, the notation $g_{-i}$ will refer to $\{g_{i'}\}_{i'=a}^b \setminus g_i$.

The set of $n$-dimensional multi-indices is $\N^n$. The sphere in $n$-dimensional space (with dimension $n-1$) is $S^{n-1} = \partial B_1^{n}$. The group of orthogonal matrices of dimension $n \times n$ is $O(n)$. The $n \times r$ Steifel manifold (matrices $Z \in \R^{n \times r}$ with $Z^T Z = I$) is $\mathbb{V}_r(\R^n)$.

The set of polynomials with real coefficients in an indeterminate value $x \in \R^n$ is $\R[x]$. For every polynomial $p \in \R[x]$, there exists a unique subset $\A \in \N^n$ and set of coefficients $\{p_\alpha \neq 0 \}_{\alpha \in \A}$ indexed by $\alpha \in \A$ such that $p(x) = \sum_{\alpha \in \A} p_\alpha x^\alpha$. 
% {\color{red}Matt: I would rather skip the non-unique set (which can always be the full set of integer tuples) and directly introduce the unique set with non zero coefficients (which is the one we should use to define the degree of the polynomial)} 
There exists a unique finite-cardinality  subset $\A'$ for each $p$ such that each $\forall \alpha \in \A': p_\alpha \neq 0.$ The degree of a multi-index is $\abs{\alpha} = \sum_i \alpha_i$. The degree of a polynomial $p$ is $\deg p = \max_{\alpha \in \A} \abs{\alpha}.$ The set of polynomials with degree at most $d$ is $\R[x]_{\leq d} \subset \R[x].$ The set of $q$-dimensional vectors of polynomials is $\R[x]^q$.

\subsection{Analysis and Measure Theory}

Refer to \cite{barvinok2002convex, tao2011introduction} for more detail about this section. The closure of a set $X \subseteq \R^n$ is $\textrm{cl}(X)$, the boundary of $X$ is $\partial(X)$, and the relative interior of $X$ is $\textrm{int}(X)$. 
The set of continuous functions over a space $X$ is $C(X)$, and its subcone of nonnegative functions over $X$ is $C_+(X) \subset C(X).$ The set of bounded measureable functions over $X$ is $B(X)$, and if $X$ is compact then $B(X) \supset C(X)$. 
% {\color{red}Matt: careful! if X is not compact, there exists unbouded continuous functions!}
The set of nonnegative Borel measures over $X$ is $\Mp{X}.$ The sets $C_+(X)$ and $\Mp{X}$ possess a duality pairing $\inp{\cdot}{\cdot}$ by Lebesgue integration with $\forall f \in C_+(X), \ \mu \in \Mp{X}: \ \inp{f}{\mu} = \int_X f(x) d \mu(x)$. This duality pairing is an inner product when $X$ is compact, for which $C_+(X)$ and $\Mp{X}$ are topological duals. The pairing $\inp{\cdot}{\cdot}$ will be extended to represent Lebesgue integration between elements of $C(X)$ and $\Mp{X}.$

A $0/1$ indicator function $I_A$ for $A \subseteq X$ takes value $I_A(x)=1$ if $x \in A$ and $I_A(x) = 0$ if $x \not\in A$. The containment relation $A \subseteq X$ implies that $\Mp{A} \subseteq \Mp{X}$ and $C(A) \supseteq C(X)$.
The $n$-dimensional volume $\textrm{Vol}_n$ of a set $A \subseteq X \subset \R^n$ is $\textrm{Vol}_n(A) = \int_A dx = \int_X I_A(x) dx$. The Radon Transform of a function $f \in B(X)$ with $\theta \in S^{n-1}, \ t \in \R$ is
\begin{align}
    \mathcal{R}f(\theta, t) = \int_{(\theta \cdot x = t) \cap X} f(x) dx. \label{eq:radon}
\end{align}
The Radon transform is an even function with $\mathcal{R}f(\theta, t) = \mathcal{R}f(-\theta, -t).$ The objective $\textrm{Vol}_{n-1}\left((\theta \cdot x = t) \cap L \right)$ from \eqref{eq:slice} may be expressed as $\mathcal{R} I_L(\theta, t)$. In practice, we will consider Radon transforms in \eqref{eq:radon} where the affine offset $t$ is restricted to lie in the bounded interval $[-R, R]$ with $R \in [0, \infty)$.

The measure of a set $A \subseteq X$ w.r.t. $\mu \in \Mp{X}$ is $\mu(A)$. The measure of $A$ may also be written as $\inp{I_A}{\mu} = \mu(A).$
The support of $\mu \in \Mp{X}$ is the set of points $x \in X$ such that every open neighborhood $N_x$ of $x$ has $\mu(N_x) > 0$. The set of measures supported in $A \subseteq X$ is $\Mp{A}$. The mass of $\mu \in \Mp{X}$ is $\mu(X) = \inp{1}{X}$, and $\mu$ is a probability measure if this mass is 1. For any two measures $\mu \in \Mp{X}, \ \nu \in \Mp{Y}$, the product measure $\mu \otimes \nu$ is the unique measure satisfying $\forall (A, B) \subset X \times Y: (\mu \otimes \nu)(A \times B)= \mu(A) \nu(B)$.

The Dirac delta $\delta_{x'}$ supported at $x' \in X$ is a probability measure satisfying $\forall f \in C(X): \inp{f}{\delta_{x'}} = f(x')$. The Lebesgue measure $\lambda_A$ of $A \subseteq X$ is the unique measure satisfying $\inp{f}{\lambda_A} = \int_{A} f(x) dx.$ The notation $\sigma_A$ will refer to the Hausdorff (surface area) measure of $\partial A$.

The pair $\mu, \nu \in \Mp{X}$ satisfies a domination relation $(\mu \geq \nu)$ if there exists a $\hat{\nu} \in \Mp{X}$ such that $\mu = \nu + \hat{\nu}$. Domination $\mu \geq \nu$ implies that $\inp{1}{\mu} \geq \inp{1}{\nu}$.

The pushforward of a map $Q: X \rightarrow Y$ along a measure $\mu \in \Mp{X}$ is the unique $Q_\# \mu$ satisfying $\forall f \in C(Y): \inp{f(y)}{Q_\# \mu(y)} = \inp{f(Q(x))}{\mu(x)}.$ The  projection map $\pi^{x}: (x, y) \mapsto x$ has a pushforward operator of $\pi^{x}_\#$. For any $\eta \in \Mp{X \times Y}$, the pushforward projection $\pi^x_\# \eta$ returns the $x$-marginal of $\eta$.

 % We will use $B_R^n$ to denote the $n$-dimensional ball with radius $R$, and $O(n)$ to refer to the $n$-dimensional group of Orthogonal matrices (coordinate frames/Stiefel manifolds). We will consider Radon transforms in which the radial limit is $R>0$ satisfies $B^n_R \supseteq L$.

\subsection{Moment-SOS Methods for Volume Computation}
% \urg{Hurwitz, Quadratic Stability, Superstability. Link superstability with the Gershgorin Circle Criterion}

% \urg{A similar overview as in Section 2.2 of \cite{tacchi2022exploiting}.}

\label{sec:prelim_volume}

Let $L \subseteq X$ be a set.
Any function $w(x) \in C(X)$ that satisfies the following pair of constraints is a continuous over-approximation to $I_L(x)$ (also written as $w \geq I_L$),
\begin{subequations}
\begin{align}
    w(x) &\geq 0 \qquad  \forall x \in X \\
    w(x) &\geq 1 \qquad  \forall x \in L.
\end{align}
\end{subequations}
The following program therefore has an infimal optimal value equal to $\textrm{Vol}_n(L)$  \cite{henrion2009approximate},
\begin{subequations}
\label{eq:vol_approx}
\begin{align}
    V^* &= \inf \int_{X} w(x) dx \\
    & w(x) \geq 0 & \forall x \in X \label{eq:vol_approx_0} \\
    & w(x) \geq 1 & \forall x \in L \label{eq:vol_approx_1}\\
    & w(x) \in C(X). \label{eq:vol_approx_c}
\end{align}
\end{subequations}

Program \eqref{eq:vol_approx} is an infinite-dimensional \ac{LP} in the variable $w(x)$, in which the affine constraints \eqref{eq:vol_approx_0} and \eqref{eq:vol_approx_1} constrain the value of $w(x)$ at each point $x \in X$. Every $w$ that is feasible for \eqref{eq:vol_approx_0}-\eqref{eq:vol_approx_c} possesses a 1-superlevel containment relation of $\{x \mid w(x) \geq 1\} \supseteq L$.

The dual of \eqref{eq:vol_approx} is  \iac{LP} with respect to the measures $\mu, \hat{\mu}$ with
\begin{subequations}
\label{eq:vol_meas}
\begin{align}
    D^* &= \sup \ \inp{1}{\mu} \\
    & \lambda_X = \mu + \hat{\mu} \\
    & \mu \in \Mp{L}, \ \hat{\mu} \in \Mp{X}.
\end{align}
\end{subequations}
The optimal solution of \eqref{eq:vol_meas} is $\mu^* = \lambda_{L}$ and $\hat{\mu}^* = \lambda_{\textrm{cl}(X-L)}$, in which $D^* = \textrm{Vol}_n(L)$. The moments of the measure $\lambda_{L}$ are generically not available in advance, otherwise the volume approximation problem would have been solved by computing $\inp{1}{\lambda_L}$. 

Approximation algorithms must be utilized in order to solve the infinite-dimensional programs \eqref{eq:vol_meas} and \eqref{eq:vol_approx} using finite-dimensional computation techniques. One such method to truncate the infinite-dimensional function nonnegativity constraints  is the moment-\ac{SOS} hierarchy of \acp{SDP} \cite{lasserre2009moments}.

A polynomial $p \in \R[x]$ is nonnegative in $\R^n$ if $\forall x \in \R^n: \ p(x) \geq 0$. A sufficient condition for nonnegativity of $p$ is if there exists a size $s \in \N$, a polynomial vector $v \in \R[x]^s$, and \iac{PSD} \textit{Gram} matrix $Q \in \psd_+^s$ such that $p(x) = v(x)^T Q v(x)$. Such a $p$ is called an \ac{SOS} polynomial, because there exists a vector $q(x) \in \R[x]^s$ with $q(x) = Q^{1/2} v(x)$ such that $p(x) =  q(x)^T q(x) = \sum_{j=1}^s q_j^2(x).$ The set of \ac{SOS} polynomials is $\Sigma[x] \in \R[x]$, and the restricted set of  \ac{SOS} polynomials with degree $\leq 2k$ is $\Sigma[x]_{\leq 2k}$ (such that $\max_j \deg q_j \leq k$). The set of \ac{SOS} polynomials equals the set of nonnegative polynomials only in the case of univariate polynomials, quadratics in any number of variables, and bivariate quartics \cite{hilbert1888darstellung}.

% \urg{Fill in \ac{SOS} background for proofs of polynomial nonnegativity}
\Iac{BSA} set $\K$ is a set described by a finite number of bounded-degree of polynomial inequality constraints $\K = \{x \mid g_i(x) \geq 0\}$. \ac{BSA} sets are closed under intersections by concatenation of constraints. Semialgebraic sets are the closures of \ac{BSA} sets under unions and projections.

A sufficient condition for a polynomial $p(x)$ to be nonnegative over $\K$ is if there exists a representation of $p(x)$ as in \cite{putinar1993compact}
\begin{subequations}
\label{eq:psatz_noeps}
\begin{align}
p(x) = &\sigma_0(x) + \textstyle \sum_i {\sigma_i(x)g_i(x)} \\
& \sigma_0, \sigma_i \in \Sigma[x], \quad \phi_j \in \R[x].
\end{align}
\end{subequations}
The \ac{WSOS} cone $\Sigma[\K]$ is the class of polynomials that have a representation in the form of \eqref{eq:psatz_noeps} w.r.t. the describing polynomials $\{g_i(x)\}$ of $\K$. The degree-restricted cone $\Sigma[\K]_{\leq 2k} \subset \Sigma[\K]$ is the set of all polynomials $p(x)$ with representation in \eqref{eq:psatz_noeps} such that $\deg \sigma_0 \leq 2k$ and $\forall i: \deg \sigma_i g_i \leq 2k$. We note that the degrees of the multiplier polynomials $\sigma$ to verify nonnegativity of $p$ over $\K$ (if they exist) could be exponential in $n$ and $\deg p$ \cite{nie2007complexity} (without the degree restriction for $\Sigma[x]_{\leq 2k}$ in place).

The set $\K$ satisfies a \textit{ball constraint} if there exists an $R \in [0, \infty)$ such that $R^2 - \norm{x}_2^2 \in \Sigma[\K]$. Satisfaction of a ball constraint is a sufficient condition for compactness of $\K$, but the converse does not always hold \cite{cimpric2011closures}. A ball constraint implies that $\Sigma[\K]$ contains the set of all positive polynomials over $\K$.

% The moment-\ac{SOS} when applied to \eqref{eq:vol_approx} is 

A degree-$k$ truncation of \eqref{eq:vol_approx} for $d \in \N$ involves restricting $w$ to $\R[x]_{\leq 2k}$, replacing constraints \eqref{eq:vol_approx_0} and \eqref{eq:vol_approx_1} by $w \in \Sigma[X]_{\leq 2k}$ and $w - 1 \in \Sigma[L]_{\leq 2k}$ respectively, and solving the resultant \ac{SDP} to obtain a bound $V_k^*$. If the \ac{BSA} sets $X$ and $L$ each satisfy a ball constraint, then the optimal values will satisfy $\lim_{k \rightarrow \infty} V_k^* = V_k.$ This general process of increasing the degree of $k$ is called the moment-\ac{SOS} hierarchy. 

Figure \ref{fig:ind_approx_over} visualizes an \ac{SOS} scheme for volume approximation for the set $L = [0.1, 0.5] \cup [0.8, 0.9]$ in the space $X = [0, 1]$. The degree-$k$ \ac{SOS} truncation of volume approximation involves a polynomial $w \in \R[x]_{\leq 2k}$ obeying the constraints $w \in \Sigma[(x(1-x))]_{\leq 2k}$, $w \in \Sigma[(x-0.1)(0.5-x)]_{\leq 2k}$ and $w -1 \in \Sigma[(x-0.8)(0.9-x)]_{\leq 2k}$. The true function $I_L$ is drawn in black in Figure \ref{fig:ind_approx_over}. The pictured \ac{SOS} polynomials $w$ have degree 6 (blue), degree 20 (red), and degree 120 (green). The resultant volume upper-bounds are documented in Table \ref{tab:value_approx}.

     \begin{figure}[!h]
         \centering
         \includegraphics[width=0.5\textwidth]{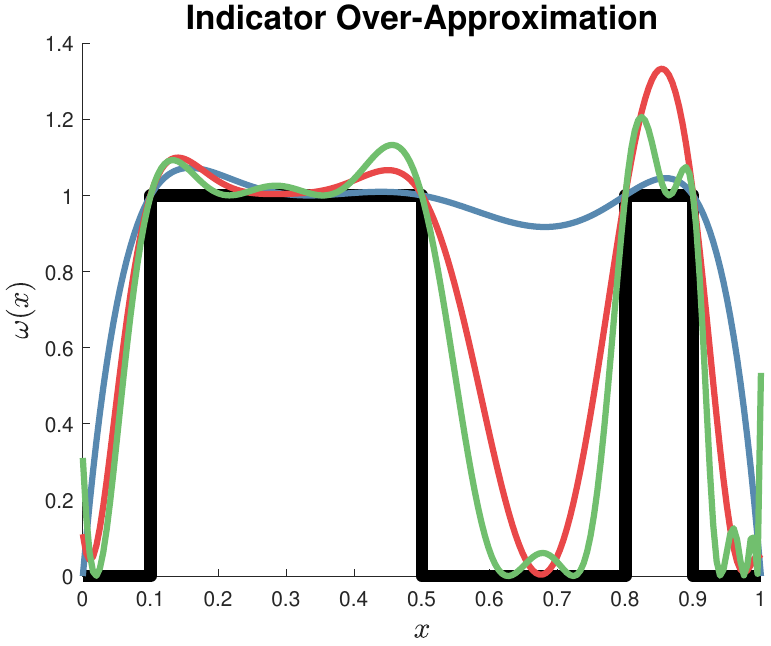}
         \caption{\label{fig:ind_approx_over}Polynomial over-approximations to $L = [0.1, 0.5] \cup [0.8, 0.9]$}       
     \end{figure}

    \begin{table}[h]
        \centering
        \caption{\label{tab:value_approx} Polynomial approximation order and volume bounds.}
        \begin{tabular}{r|c c c c}
            Function &  $k=6$ & $k=20$ & $k=120$ & Truth\\
            Volume (upper)& 0.9269 & 0.7472 & 0.6740 &  0.5 \\
        \end{tabular}

    \end{table}

Even though the 1-sublevel sets of Figure \ref{fig:ind_approx_over} are very close to $L$, the volume bounds in Table \ref{tab:value_approx} are far from the true value of 0.5. Unfortunately, there does not exist an efficient method to compute the volume of polynomial sublevel sets purely from the polynomial's coefficients \cite{guthrie2022inner}.
This large gap in volume estimates is due to an oscillatory Gibbs phenomenon \cite{gottlieb1997gibbs} involved in the approximation of the discontinuous $I_L$ by the set of continuous polynomials. The presence of Gibbs phenomena contributes towards an asymptotic convergence rate of $O(1/\log \log k)$ in the error of volume approximation \cite{korda2018convergence}.

Stokes methods (introduced further in Section \ref{sec:stokes}) \cite{tacchi2023stokes} can be used to refine volume estimates and yield faster convergence in $k$ through the addition of redundant derivative constraints.

% \urg{TODO: now comment about convergence rates, I think \eqref{eq:vol_approx} is $\log \log d$, but I don't remember the correct source. }

% Every $p(x)$ that can be expressed in \eqref{eq:psatz_noeps} is a member of the \ac{WSOS} cone $\Sigma$

% Putinar Psatz \cite{putinar1993compact}

% Scherer Psatz for Matrices \cite{scherer2006matrix}
\section{Slice-Volume}
\label{sec:slice_volume}

This section will solve the slice-volume Problem \ref{prob:slice} using infinite-dimensional \acp{LP}. We will begin with the following assumption:
% The following assumptions will be posed in this section:
\begin{itemize}
    \item[A1] There exists an $R  \in [0, \infty)$ such that $L \subseteq B_R^n$.
\end{itemize}

The ball constraint $R^2 - \norm{x}^2$ can be added to 
the description of any \ac{BSA} $L$ satisfying A1 without changing its geometry.

% \subsection{Assumptions}

\subsection{Slice-Volume Variables and Supports}

The variables used in the slice-volume optimization problem are listed in Table \ref{tab:slice_variables}.  
% The Radon transform will be expressed in terms of $(\theta, t)$.
\begin{table}[h]
    \centering
    \caption{Variables in slice-volume analysis}
    \begin{tabular}{c|l}         
          $\theta$ & Direction \\
        $t$ & Affine Offset \\
$Z$ & Local Coordinate Frame \\
$y$ & Local Coordinate on the plane $\theta \cdot x = t$
\end{tabular}    
    \label{tab:slice_variables}
\end{table}

The support sets involving variables in Table \ref{tab:slice_variables} are
\begin{subequations}
\label{eq:supp_sets}
\begin{align}
    \Omega &= \{(\theta, t) \in S^n \times [-R, R]\} \label{eq:supp_omega}\\
        \Omega_Z &= \{(\theta, t, Z) \in S^n \times [-R, R] \times \R^{n \times (n-1)} \mid [\theta, Z] \in O(n)\} \label{eq:supp_omega_Z}\\
        \Psi &= \Omega_Z \times B_R^{n-1} \label{eq:supp_psi_all} \\
    \Psi_L &= \Psi \cap \{(\theta t + Z y) \in L\} \label{eq:supp_psi}
\end{align}
\end{subequations}
The domain of the $R$-truncated Radon transform \eqref{eq:radon} is $\Omega$ from \eqref{eq:supp_omega}. The set $\Psi_L$ represents $L$ via the over-complete formulation of $x = \theta t + Z y$. Constraints \eqref{eq:supp_omega_Z} and  \eqref{eq:supp_psi} implies that $Z \in \mathbb{V}_{n-1}(\R^n)$. 

Figure \ref{fig:slice_arrows} visualizes the slicing geometry and variables from Table \ref{tab:slice_variables}. The set $L$ is the region inside the outer gray ellipsoid and outside the inner red ellipsoid. The green arrow is the slicing direction $\theta$, and the two black arrows are the columns of $Z$. The value $t$ is the affine offset in the direction of $\theta$ between the origin (center of the coordinate frame) and the blue slicing plane. In this manner, a point $x \in L$ (black asterisk) may be represented by an appropriate choice of $x = \theta t + Z y.$

\begin{figure}[h]
    \centering
    \includegraphics[width=0.5\linewidth]{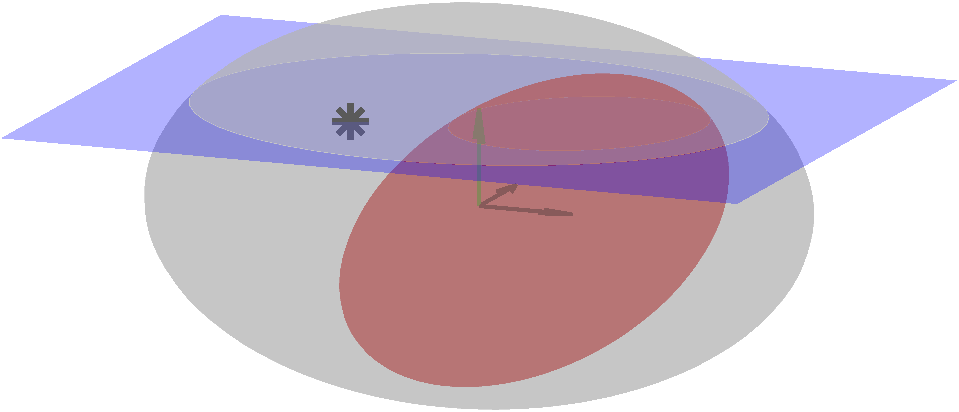}
    \caption{Slicing geometry and variables from Table \ref{tab:slice_variables}.}
    \label{fig:slice_arrows}
\end{figure}

% \urg{need to make a figure here to explain Table \ref{tab:slice_variables}.}

% The matrix $[\theta, \, Z]$ provides an orthogonal coordinate frame, for which $t$ is the affine offset away from $0$ and $y$ 

% The Radon transform coordinates $(\theta, t)$ has a point with local coordinates $y$ and sphere-tangent coordinate frame $Z$ in the set $L$ if $(\theta, t, Z, y) \in \Psi_L$ from \eqref{eq:supp_psi}.

% The projection $\Psi_L$ is the set of all $(t, \theta, y)$ such that there exists 

\subsection{Slice-Volume Function Program}

A functional \ac{LP} for the slice-volume task will be developed based on the smooth-indicator-approximating volume computation methodology from Section \ref{sec:prelim_volume}.
For notational convenience, let us use define the $y$-integration $\Lambda_R: C(\Psi) \rightarrow C(\Omega_Z)$ as $\forall w(\theta, t, Z, y) \in C(\Psi)$:
\begin{align}
    \Lambda_R w(\theta, t, Z) = \int_{B_{R}^{n-1} } w(\theta, t, Z, y) dy. \label{eq:integrate_y}
\end{align}

We introduce a continuous auxiliary function variable $w(\theta, t, Z, y)$ in order to pose the following slice-volume maximization program:
\begin{subequations}
\label{eq:slice_upper}
\begin{align}
    p^* = &\inf_{\gamma \in \R} \quad \gamma \\
    & \gamma \geq  \Lambda_R w(\theta, t, Z) & & \forall (\theta, t, Z) \in \Omega_Z  \label{eq:slice_upper_int}\\
    & w(\theta, t, Z, y) \geq 1 & & \forall (\theta, t, Z, y) \in \Psi_L \label{eq:slice_upper_psi}\\
    & w(\theta, t, Z, y) \geq 0 & &\forall (\theta, t, Z, y) \in \Psi \label{eq:slice_upper_psi_comp} \\
    & w(\theta, t, Z, y) \in C(\Psi) \label{eq:slice_upper_w}.
\end{align}
\end{subequations}

\begin{lem}
\label{lem:upper_bound}
    Program \eqref{eq:slice_upper} is an upper-bound on Problem \ref{eq:slice} ($p^* \geq P^*$).
\end{lem}
\begin{proof}

The indicator function $w^*(\theta, t, Z, y) = I_\Psi$ will satisfy constraints \eqref{eq:slice_upper_psi}-\eqref{eq:slice_upper_psi_comp} while ignoring the continuity requirement in \eqref{eq:slice_upper_w}.  The integrated term $\Lambda_R w^*(\theta, t, Z)$ may be interpreted as the radon transform $\mathcal{R} I_L(\theta, t)$ with respect to an arbitrary coordinate frame $Z \in \mathbb{V}_{n-1}(\R^n)$. Every value $\gamma$ such that $\forall (\theta, t, Z) \in \Omega_Z: \  \gamma \geq \mathcal{R} I_L(\theta, t)$ \eqref{eq:slice_upper_int} is an upper-bound on the maximal slice volume. 

Letting $w$ be any continuous feasible solution \eqref{eq:slice_upper_psi}-\eqref{eq:slice_upper_w}, it holds that $w \geq w^*$ pointwise and also that $\Lambda_R w(\theta, t, Z, y) \geq \mathcal{R}I_L(\theta, t)$ at each $(\theta, t) \in \Omega$. It therefore holds that $p^* \geq P^*$.

\end{proof}

In order to prove equality between Problem \ref{prob:slice} and the infimum of \eqref{eq:slice_upper}, we require a notion of almost uniform convergence \cite{ash1972real, henrion2011inner}:

\begin{defn}
\label{defn:almost_uniform}
    A sequence of functions $\{f_n\}_n$ converges to $f$ \textbf{almost uniformly} in a set $X$ if for every $\epsilon>0$, there exists $A \subset X$ such that $\textrm{Vol}(A) < \epsilon$ and $f_n \rightarrow f$ uniformly on $X \setminus A$.
\end{defn}

\begin{thm}
\label{thm:slice_no_relaxation}
    Under A1, there is no relaxation gap between Problem \ref{prob:slice} and Program \eqref{eq:slice_upper} ($p^* = P^*$).
\end{thm}
\begin{proof}

The set $C(\Psi)$ is dense in $L^1(\Psi)$, implying the existence of a sequence of functions $w^k(\theta, t, Z, y) \in C(\Psi)$  converging to $w^*$ from above 
in an $L^1$ sense:
\begin{align}
    w^k \geq w^*, \qquad   \lim_{k\rightarrow \infty} \int_{\Psi} \left(w^k(\theta, t, Z, y) - w^*(\theta, t, Z, y) \right) d\theta \; dt \; dZ \; dy  = 0.
\end{align}

By Theorem 2.5.3 of \cite{ash1972real}, $L_1$-convergence of $w^k$ implies the  existence a subsequence $w^{k'}$ converging to $w^*$ almost uniformly.

It therefore holds that $\{\Lambda_R w_k(\theta, t, Z) \}_k$ converges almost uniformly to $\mathcal{R} I_L(\theta, t)$. The pointwise upper-bound of $w^k \geq w^*$ ensures that the infimum will therefore be maintained, though a minimum may not be met (in general).

% \urg{Could we use Lemma 2.5.4 ($\lambda_{\Omega \times B_R^{n-1}}$ finite + function construction $\rightarrow$ a.e. convergence) and Theorem 2.5.5. (Egoroff's theorem, $\lambda_{\Omega \times B_R^{n-1}}$ finite + a.e. convergence $\rightarrow$ almost uniform convergence) of \cite{ash1972real}?}

% This implies that $\int_{B^{n-1}_R} w_k(\theta, t, y) dy \rightarrow \mathcal{R}_L(\theta, t)$ pointwise, given that  \urg{TODO: check this, is this true? I could use help on the convergence aspect. We should be able to argue that $\lambda_L$ has a density with respect to $\lambda_R^{n-1}.$}

% Every $w^k$ will produce a bound $p^k = \inf_{(\theta, t) \in \Omega} \int_{B^{n-1}_R} w(\theta, t, y) dy$ such that $p^k \geq p^*$.

% Then the quantities $\inf_{(\theta, t)}$

% \urg{Finish up the convergence proof. I'm getting confused on this point. It is probably simple to prove infima but i'm missing the right theorem.}

% and the indicator function $w^*(\theta, t, y) = I_\psi$
    
\end{proof}

% \begin{rem}
%     The infimal result of \ref{thm:slice_no_relaxation} is retained when enforcing that constraints \eqref{eq:slice_upper_psi}-\eqref{eq:slice_upper_psi_comp} are strict inequalities.
%     \label{rmk:strict_ineq}
% \end{rem}

\begin{cor} The objective value $p^*$ from \eqref{eq:slice_upper} is finite under A1.
\end{cor}
\begin{proof}
    \textbf{Bounded from above:} The function $w(\theta, t, Z, y) = 1$ is feasible for constraints \eqref{eq:slice_upper_psi} - \eqref{eq:slice_upper_w}. The choice of $\gamma = \Lambda_R(1) = \int_{B_R^{n-1}} 1 dy = \textrm{Vol}_{n-1}(B_R^{n-1})$ is feasible for \eqref{eq:slice_upper_int}, given that $R$ is finite by A1.

    \textbf{Bounded from below:} In the case where $L=\varnothing$, the function $w = 0$ would be the minimal possible choice satisfying \eqref{eq:slice_upper_psi} - \eqref{eq:slice_upper_w}. 

    As such, the optimal value $p^*$ is bounded between $[0, \textrm{Vol}_{n-1}(B_R^{n-1})].$
\end{proof}

\begin{exmp}
    This pathological example will demonstrate almost uniform convergence of the slice-volume construction. In this example, we will restrict $t=0$ to aid in explanation. Consider the 1-dimensional disc set $L = \{x \in \R^2 \mid x_2 = 0 \} \cap B_R^2$ with $R \in (0, \infty)$. The discontinuous Radon transform of the indicator function of $L$ is 
    \begin{equation}
        \mathcal{R} I_L(\theta, 0) = \begin{cases}
            2R & \theta = \pm (0, 1) \\
            0 & \textrm{else}.
        \end{cases}
    \end{equation}

    Let $\eta \in (0, R\sqrt{2}/2)$ be a tolerance. We can construct a sequence of functions $f_\eta(x) \geq I_L(x)$ by defining
    \begin{equation}
        f_\eta(x) = \begin{cases}
            1 - \abs{x_2}/\eta & \abs{x_2} \leq \eta \\
            0 & \textrm{else}.
        \end{cases}
    \end{equation}

    The planar direction $\theta \in S^1$ will be parameterized by an angle $\phi \in [0, 2 \pi]$ under the relation $\theta = [\cos(\phi+\pi/2), \sin(\phi + \pi/2)]$.
    The Radon transform of $f_\eta$ at the translate $t=0$ in the domain $\phi \in [0, \pi/2]$ is
    \begin{equation}
                \mathcal{R}f_\eta(\phi, 0) = \begin{cases}
            \eta/\cos(\pi/2-\phi) & \phi \in [\pi/4, \pi/2]
             \\
            \eta/\sin{\phi} & \phi \in [\sin^{-1}(\eta/R), \pi/4] \\
            R(2-(R/\eta) \sin{\phi}) & \phi \in [0, \sin^{-1}(\eta/R)].
        \end{cases} \label{eq:radon_pathologic}
    \end{equation}

The transform in \eqref{eq:radon_pathologic} can be extended to $[0, 2\pi]$ by symmetry. Under this definition for $f_\eta$,
\begin{align}
    \mathcal{R}f_\eta(\phi, 0) &= 2R,& \mathcal{R}f_\eta(\sin^{-1}(\eta/R), 0) &= R,& \mathcal{R}f_\eta(\pi/2, 0) &= \eta.
\end{align}

For any $\epsilon' > 0$, we can compute the angle of $\phi' = \sin^{-1}(\eta/\epsilon')$. The volume of $\phi \in [0, 2\pi]$ with $\mathcal{R}f_\eta (\phi, 0) > \epsilon'$ is $4 \phi'$. Given an $\epsilon > 0$ with $\epsilon' = \epsilon/4$, we can choose a $\eta$  and a $\phi'$ such that $4 \phi' \leq \epsilon$. This provides almost uniform convergence, because $\eta$ and $f_\eta$ can be chosen to ensure that the volume $4\phi' \rightarrow 0$  as $\epsilon \rightarrow 0$.

\end{exmp}

% The function $q(\theta, t) = \int_{B_{R}^{n-1} } v(\theta, t, y) dy$ is an upper-bound of the slice volume $\textrm{Vol}_{n-1}(\theta^T x = t) \cap L$.  

            % mom_num = prod(gamma.((mom_curr .+ 1)./2));
            % mom_denom = gamma.(1+(n+sum(mom_curr))./2);
            % mom_scale = R^(n + sum(mom_curr));
            % moment[i] = mom_scale .* mom_num ./ mom_denom;

\subsection{Slice-Volume Measure Program}

A dual program to \eqref{eq:slice_upper} can be derived by introducing measure variables,
\begin{table}[h]
    \centering
    \caption{Variables in slice-volume measure \ac{LP}}
    \begin{tabular}{l|l}         
          $\mu_0 \in \Mp{\Omega_Z}$ & Choice of slicing direction $(t, \theta)$ and coordinate frame $Z$ \\
        $\mu \in \Mp{\Psi_L}$ & Volume tracking\\
$\hat{\mu} \in \Mp{\Psi}$ & Slack measure \\
\end{tabular}    
    \label{tab:slice_variables_meas}
\end{table}

Define $\lambda_R^{n-1}$ as the Lebesgue measure of the ball $B^n_R$.
% The Lagrangian of \eqref{eq:slice_upper} is
% \begin{subequations}
% \begin{align}
%     \scL &= \gamma - \inp{\gamma - \int w(\theta, t, y) d \lambda_R^{n-1}(y)}{\mu_0} - \inp{w-1}{\mu} - \inp{w}{\hat{\mu}} \\
%     &= \inp{1}{\mu} + \gamma(1-\inp{1}{\mu_0}) + \inp{w}{\mu_0 \otimes \lambda^{n-1}_B - \pi^{\theta t y}_\# \mu - \hat{\mu}}
% \end{align}
% \end{subequations}
A convex measure \ac{LP} that relaxes Problem \ref{prob:slice} is:
% of \eqref{eq:slice_upper} is
\begin{subequations}
\label{eq:slice_upper_meas}
\begin{align}
    m^* = &\sup \  \inp{1}{\mu} \label{eq:slice_upper_meas_obj} \\
    &\mu_0 \otimes \lambda_{R}^{n-1} =  \mu + \hat{\mu} \label{eq:volume_complement}\\
    & \inp{1}{\mu_0} = 1 \label{eq:slice_upper_meas_prob}\\
    & \mu_0 \in \Mp{\Omega_Z}, \\
    & \hat{\mu} \in  \Mp{\Psi} \\
    & \mu \in \Mp{\Psi_L} \label{eq:slice_upper_meas_L}.
\end{align}
\end{subequations}

The measure $\mu_0$ is a probability distribution over $(\theta, t, Z)$ that picks out the best direction $\theta$,  offset $t$, and coordinate frame $Z$ for a slice. 
% The measure $\pi^{\theta t y}_\# \mu $ is the mass distribution of the slice volume with complement $\hat{\mu}$ from \eqref{eq:volume_complement}. 
The objective  in \eqref{eq:slice_upper_meas_obj} is the supremal slice-volume.

\begin{thm}
\label{thm:slice_meas_upper_bound}
    Program \eqref{eq:slice_upper_meas} is an upper-bound on Problem \ref{eq:slice} with $m^* \geq P^*$.
    \label{thm:slice_upper_meas}
\end{thm}
\begin{proof}
    Let $(\theta', t') \in \Omega$ be a slicing direction and offset. We will now construct measures from Table \ref{tab:slice_variables_meas} from $(\theta, t)$ that are feasible for \eqref{eq:volume_complement}-\eqref{eq:slice_upper_meas_L}. Let $Z' \in \R^{(n-1) \times n}$ be a matrix such that $[\theta', \, Z'] \in O(n)$.    
    We let $\mu_0' = \delta_{\theta=\theta'} \otimes \delta_{t=t'} \otimes \delta_{Z = Z'}$ be a probability distribution representing the slicing choice $(\theta', t', Z')$.  
    
    We can construct a set $Y = \{y \in B^{n-1}_R \mid \theta' t' + Z' y \in L\}$ that has a Lebesgue measure $\lambda_Y \leq \lambda_R^{n-1}$. The measures of $\mu' = \mu_0' \otimes \lambda_Y$ and $\hat{\mu}' = \mu_0' \otimes  \otimes \lambda_{\textrm{cl}(B^{n-1}_R - Y)}$ can therefore be defined to satisfy the domination relation \eqref{eq:volume_complement}.

    A measure solution $(\mu_0', \hat{\mu}', \mu')$ therefore exists for every possible choice of $(\theta', t') \in \Omega$. The feasible set of measures is therefore a superset of the (induced) feasible set of $\Omega, \Omega$, implying that $m^* \geq P^*$.    
\end{proof}

\begin{lem}
    The measure solutions to \eqref{eq:slice_upper_meas} are bounded under A1.
\label{lem:meas_bounded}
\end{lem}
\begin{proof}
    Boundedness of nonnegative measures will be proven using the sufficient combination of compact support and finite mass.
Assumption A1 ensures that $L$ is compact with finite $R$. The set $\Omega_Z = S^{n-1}\times [-R, R] \times \mathbb{V}_{n-1}(\R^n)$ is compact, as is the set $\Psi$. The set $\Psi_L$ is contained within the compact set $\Psi$, since  $\Psi_L$ is the preimage of the compact set $L$ under the continuous mapping $\theta t + Z y$.

The mass of $\mu_0$ is set to 1 by  \eqref{eq:slice_upper_meas_prob}. 
\begin{align}
    \inp{1}{\mu} + \inp{1}{\hat{\mu}} = \inp{1}{\mu_0 \otimes \lambda_R^{n-1}} = (1)(\textrm{Vol}_{n-1}{B_R^{n-1}}) = \textrm{Vol}_{n-1}{B_R^{n-1}} < \infty.
\end{align}

Because $\mu$ and $\hat{\mu}$ are each nonnegative measures, their masses are likewise nonnegative numbers. Their masses therefore satisfy $\inp{1}{\mu}, \inp{1}{\hat{\mu}} \in [0, \textrm{Vol}_{n-1}{B_R^{n-1}}].$
    
\end{proof}

\begin{thm}
\label{thm:strong_duality_indicator}
    Strong duality holds between programs \eqref{eq:slice_upper} and \eqref{eq:slice_upper_meas} under A1.
\end{thm}
\begin{proof}
    See Appendix \ref{app:duality}.
\end{proof}

% &= $\mu_0 \in \Mp{\Omega}, \ \hat{\mu} \in \Mp{\Omega \times B_R^{n-1}}, $ and $\mu \in \Mp{\Psi_L}
\section{Slice-Volume Semidefinite Programs}
\label{sec:slice_sos}

This section will apply the moment-\ac{SOS} hierarchy of \ac{SDP} towards solution to \eqref{eq:slice_upper}.

\subsection{Slice-Volume SOS Program}

We will restrict $w$ from \eqref{eq:slice_upper_w} to be a polynomial $w \in \R[\theta, t, Z, y]_{\leq 2k}$ for some degree $k \in \N$. The following proposition is required in order to evaluate the integral from \eqref{eq:slice_upper_int}:
\begin{prop}[Theorem 3.1 of \cite{lasserre2001solving}]
The moments of the Lebesgue distribution for the radius-$R$ ball in $n$ dimensions are
\begin{align}
    \forall \alpha \in \N^n: \qquad \int_{B^n} y^\alpha d y = R^{n + \abs{\alpha}} \frac{\prod_i \Gamma((\alpha_i + 1)/2)}{\Gamma(1 + (n+\abs{\alpha})/2)}.
\end{align}
\end{prop}

% Let $\Lambda^{n-1}_R: \R[\theta, t, y] \rightarrow \R[\theta, t]$ be the unique linear mapping conducted by $w(\theta, t, y) \mapsto \int_{B^{n-1}_R} w(\theta, t, y) dy$. 

We will now be assuming that $L$ has \iac{BSA} structure:
\begin{itemize}
    \item[A2] The set $L$ is  \iac{BSA} set that includes a ball constraint description $R^2 -\norm{x}_2^2 \geq 0$.
\end{itemize}

\begin{prob}
    The degree-$k$ \ac{SOS} truncation of \eqref{eq:slice_upper} is:
\begin{subequations}
\label{eq:slice_upper_sos}
\begin{align}
    p^*_k = &\inf_{\gamma \in \R} \quad \gamma \\
    & \gamma - \textstyle \Lambda_R w(\theta, t, Z)  \in \Sigma[\Omega_Z]_{\leq 2k}\label{eq:slice_upper_sos_int}\\
    & w(\theta, t, Z, y) - 1  \in \Sigma[\Psi_L]_{\leq 2k} \label{eq:slice_upper_sos_psi}\\
    & w(\theta, t, Z, y) \in \Sigma[\Psi]_{\leq 2k} \label{eq:slice_upper_sos_psi_comp} \\
    & w(\theta, t, Z, y) \in \R[\theta, t, Z, y]_{\leq 2k}\label{eq:slice_upper_sos_w}.
\end{align}
\end{subequations}    
\end{prob}

\begin{thm}
\label{thm:sos_indicator_slice}
    Program \eqref{eq:slice_upper_sos}  provides a decreasing sequence of upper bounds $p_k^* \geq p^*_{k+1} \geq \ldots$ that converges as $\lim_{k \rightarrow \infty} p_k^* = P^*$ under assumptions A1 and A2.
\end{thm}
\begin{proof}
See Appendix \ref{app:poly_approx}.
\end{proof}

\begin{rem}

The \ac{WSOS} Constraint \eqref{eq:slice_upper_sos_int} 
has a polynomial $\sigma_0(\theta, t, Z) \in \Sigma[\theta, t, Z]_{\leq 2k}$ in representation \eqref{eq:psatz_noeps}. The Gram matrix of $\sigma_0$ is a \ac{PSD} variable in the \ac{SDP} formulation of problem \eqref{eq:slice_upper_sos}. The dual \ac{PSD} variable to this Gram matrix  is associated with the moment matrix of $\mu_0$ \cite{lasserre2009moments}. If the output of \iac{SDP} solver for \eqref{eq:slice_upper_sos} yields a low-rank dual variable (candidate moment matrix for $\mu_0$), then the matrix-factorization-based methods of \cite{henrion2005detecting2} can be used to attempt extraction of (an orbit of) candidates $(\theta, t, Z)$ to solve \eqref{eq:slice_upper_meas}.

\end{rem}

\subsection{Extension to Semialgebraic Sets}

% \begin{rem}
    Assume that $L$ may be expressed as the union of $N_c$ \ac{BSA} sets $L = \cup_{j=1}^{N_c} L_j$, in which each $L_j$ possesses $(R^2 - \norm{x}_2^2 \geq 0)$ in its \ac{BSA} representation. 
    Constraint \eqref{eq:slice_upper_sos_int} can be generalized to $\forall j \in 1..N_c: \   w(\theta, t, y) -1 \in \Sigma[\Psi_{L_j}]_{\leq 2k}$ in the union case, while retaining convergence as $\lim_{k \rightarrow \infty} p_k^* = P^*.$
% \end{rem}

% \begin{rem}
    In this manner, lower-bounds for the minimum slice-volume of $L$ can be found by maximizing the slice-volume over the complement of $L$ (which can be represented as a union of \ac{BSA} sets).
% \end{rem}

% \begin{rem}
    Slice-volumes can also be computed for semialgebraic sets described by given projections. Let $L \subset X$ and $\hat{L} \subset X \times E$ be sets such that $L = \pi^x \hat{L}$ and $\hat{L}$ is \ac{BSA}. This implies that $L$ is semialgebraic (projection of \iac{BSA} set). Constraint \eqref{eq:slice_upper_sos_psi} can then be implemented as \iac{WSOS} constraint over the following \ac{BSA} set:
\begin{align}
        \Psi_{\hat{L}} &= \{(\theta, t, Z, y, e) \in  \Psi \times E\mid  \exists e: (\theta t + Z y) \in \hat{L}, \ [\theta, \, Z] \in O(n)\}. \label{eq:supp_psi_hat}
        \end{align}
% \end{rem}

\subsection{Computational Complexity}
\label{sec:complexity}

This subsection will quantify the computational complexity of solving \ac{SDP} formulations of \eqref{eq:slice_upper_sos} via  interior point methods. The per-iteration scaling complexity of solving \iac{SDP} arising from the moment-\ac{SOS} hierarchy in $N$ variables at degree $k$ (with a single \ac{PSD} block) is $O(N^{6k})$ and $O(k^{4N})$ \cite{lasserre2009moments, miller2022eiv_short}.

The computational complexity of degree-$d$ SOS tightenings of programs \eqref{eq:slice_upper} is dominated by the \ac{PSD} matrix constraints present in \eqref{eq:slice_upper_sos_psi} and \eqref{eq:slice_upper_sos_psi_comp}. These constraints involve the $(n)+(1)+(n-1)+n(n-1) = n^2 + n$ variables $(\theta, t, Z, y)$. The maximal-size \ac{PSD} (Gram) matrices will have size $\binom{n^2+n+k}{k}$, which quickly grows intractable as $n$ increases. In comparison, constraints \eqref{eq:slice_upper_sos_int} have maximal \ac{PSD} sizes of $\binom{n+1+k}{k}$ and $\binom{2n+k}{k}$, respectively.

% \urg{make more formal, add more detail.}
\section{Reduction of Computational Complexity}

\label{sec:reduce_complex}

These subsequent sections will detail how algebraic and symmetric structures inside the support sets \eqref{eq:supp_sets} can be used to reduce the size of the Gram matrices in Section \ref{sec:complexity}.

\subsection{Dimension-Independent}

\label{sec:structure_dim_indep}
The slice volume programs offer several forms of extant structure. 

\subsubsection{Discrete Symmetry}
\label{sec:symmetry_discrete}
To begin with, the Radon transform of a function is symmetric under the exchange $(\theta, t) \leftrightarrow (-\theta, -t)$ (in domain $\Omega$). As a result, the functions $(w, \tilde{w})$ may be chosen to be even under the exchange $(\theta, t, Z, y) \leftrightarrow (-\theta, -t, -Z, -y)$ (in domain $\Psi$). This is an instance of a sign-symmetry, which can be exploited through the \ac{SOS} methods developed in  \cite{gatermann2004symmetry, lofberg2009pre}. Symmetries perform a block-diagonalization of the Gram matrices, yielding a sequence of smaller \ac{PSD} constraints while retaining the same degree-$k$ optimal value.

\subsubsection{Continuous Symmetry}
\label{sec:symmetry_continuous}
The set $\Psi_L$ admits an $O(n-1)$ symmetry. For any $Q \in O(n-1)$ and $(\theta, t, Z, y) \in \Psi_L$, it holds that $(\theta, t, Z Q^T , Qy) \in \Psi_L$ (because $\theta t + Z y = \theta t + (Z Q^T) (Q y)$). The function $w(\theta, t, Z, y)$ can be chosen to respect this orthogonal symmetry without changing the objective value of \eqref{eq:slice_upper}, because the Radon transform function $\mathscr{R} I_L(\theta, t)$ used in the proof of lemma \ref{lem:upper_bound} is $Z$-independent.

Further symmetries can be exploited if $L$ also possesses geometric symmetries (e.g., dihedral).

% Furthermore, orthonormality of $Q$ ensures that $\int_{B^{n-1}_R} w(\theta, t, Z Q^T, Q y)$

% The multiplication $Z y$ inside \eqref{eq:supp_psi} is invariant under the Orthogonal transformation $P \in O(n-1): \ Z y = (Z P) (P^T y) = \tilde{Z} \tilde{y}$. Such a right-multiplication action on $Z$ fixes the orthogonality constraint:
% \begin{align}
% \label{eq:sym_orthogonal_theta_Z}
%     [\theta, Z]  \in O(n) \implies \forall P \in O(n-1): \ [\theta, Z] \begin{bmatrix} 1 & \0 \\ \0 & P \end{bmatrix} = [\theta, \tilde{Z}] \in O(n).
% \end{align}

% The set $\Omega$ has the description of $\{(\theta, t) \mid \norm{\theta}_2^2 -1 = 0, \ R^2-t^2 \geq 0\}$. 

% Gr\"{o}bner basis reduction methods can be used 

\subsubsection{Quotient Reduction}

The impositions of $\norm{\theta}_2^2 = 1$ and $[\theta, Z]$ in the descriptions of $\Omega$ and $\Psi_L$ yield a sequence of degree-2 polynomial equality constraints. Gr\"{o}bner basis reduction methods \cite{cox2013ideals} can be used to further reduce the size of the Gram matrices \cite{parrilo2003exploiting, parrilo2005exploiting}. As an example, the replacement rule of  $\theta_n^2 \mapsto 1-\sum_{i=1}^{n-1} \theta_i^2$ would be evaluated anywhere $\theta_n^2$ appears in a \ac{WSOS} polynomial expression. It is therefore not necessary for the polynomial $v$ or any \ac{WSOS} multiplier to contain a monomial with a degree-$\geq 2$ term in $\theta_n$. 
% Quotient reduction can be incorporated into symmetry reduction 

% whose algebraic structure may be realized using Grobner elimination.

% Harmonic hierarchies may be used to entirely eliminate the variables $\theta$ and $(\theta, Z)$, leaving only the   $n+1$ interior variables $(t, y)$ \cite{cristancho2022harmonic}. Such elimination is contingent on extending the Harmonic hierarchy method towards sphere-constrained problems with additional support constraints, which could ruin the delicate doubly-transitive structure of $O(n)$ (Harmonic hierarchies might not be feasible in these general settings).

\subsection{Dimension 2, 4, 8}
\label{sec:dim_248}
A manifold $\mathcal{M}$ is \textit{parallelizable} if there exists a continuous map between  $\mathcal{M}$ and the space of coordinate frames on $\mathcal{M}$ (the tangent bundle of $\mathcal{M}$ is trivial).

The sphere $S^{n'}$ is parallelizable only in dimensions $n'=\{1, 3, 7\}$ by Corollary 2 of \cite{bott1958spheres}. These choices correspond to the division rings defined at $n=n'+1 = \{2, 4, 8\}$ of complex numbers, quaternions, and octonions. The basis for the tangent space may be constructed by forming multiplication tables of these division rings \cite{bott1958spheres}.

In dimension $n=2$ with $\theta \in S^1$ corresponding to complex numbers $\theta_1 + i \theta_2$, the frame $Z$ satisfies
\begin{subequations}
\label{eq:Z_division}
\begin{align}
    [\theta, Z] &= \begin{bmatrix}
        \theta_1 & -\theta_2 \\
        \theta_2 & \theta_1
    \end{bmatrix}. \label{eq:Z_division2}
\intertext{The process can be repeated for quaternions $\theta = \theta_1 + \theta_2 i + \theta_3 j + \theta_4 k \in S^3$ as}
        [\theta, Z] &= \begin{bmatrix}
        \theta_1 & -\theta_2 & -\theta_3 & \theta_4\\
        \theta_2 & \theta_1 & -\theta_4 & -\theta_3\\
        \theta_3 & -\theta_4 & \theta_1 & -\theta_2\\
        \theta_4 & \theta_3 & \theta_2 & \theta_1
    \end{bmatrix},\label{eq:Z_division4}
    \intertext{and similarly for octonions with }
        [\theta, Z] &= \begin{bmatrix}
               \theta_1 & -\theta_2 & -\theta_3 & \theta_4 & -\theta_5 & \theta_6 & \theta_7 & -\theta_8\\
        \theta_2 & \theta_1 & -\theta_4 & -\theta_3 & -\theta_6 & -\theta_5 & \theta_8 & \theta_7\\
        \theta_3 & -\theta_4 & \theta_1 & -\theta_2 & -\theta_7 & \theta_8 & -\theta_5 & \theta_6\\
        \theta_4 & \theta_3 & \theta_2 & \theta_1 & -\theta_8 & -\theta_7 & -\theta_6 & -\theta_5 \\
        \theta_5 & -\theta_6 & -\theta_7 & \theta_8 & \theta_1 & -\theta_2 & -\theta_3 & \theta_4 \\
        \theta_6 & \theta_5 & -\theta_8 & -\theta_7 & \theta_2 & \theta_1 & -\theta_4 & -\theta_3\\
        \theta_7 & -\theta_8 & \theta_5 & -\theta_6 
		& \theta_3 & -\theta_4 & \theta_1 & -\theta_2\\
        \theta_8 & \theta_7 & \theta_6 & \theta_5 & \theta_4 & \theta_3 & \theta_2 & \theta_1
    \end{bmatrix}. \label{eq:Z_division8}
\end{align}
\end{subequations}

Substitutions in \eqref{eq:Z_division} form a (nonunique) explicit representation of $Z$ in terms of $\theta$. 
% This nonuniqueness of coordinate frame representation is encoded in the $O(n-1)$ action in \eqref{eq:sym_orthogonal_theta_Z}. 
As a result, constraints such as in \eqref{eq:slice_upper_psi} will only contain the $2n$  variables $(\theta, t, y)$ (down from $n^2 + n$). The maximal-size \ac{PSD} matrix of \eqref{eq:slice_upper_sos_psi} is reduced from $\binom{n^2+n+k}{k}$ to $\binom{2n+k}{k}$ by virtue of eliminating $Z$ when $n \in \{2, 4, 8\}$.

The specific selection of a $Z$ matrix in terms of $\theta$ will destroy the continuous $O(n-1)$ symmetry from Section \ref{sec:symmetry_continuous}. Breaking the $O(n-1)$ symmetry is acceptable, due to the elimination of $Z$ in the count of the number of variables.

\subsection{Dimension 3}

\label{sec:dim3}
The sphere $S^2$ (embedded in 3-dimensional Euclidean space) is not parallelizable \cite{bott1958spheres}, and therefore $Z$ cannot be assigned to $\theta$ in a continuous manner. Such a topological obstruction prevents the definition of a polynomial-valued map $\theta \rightarrow Z$ from \eqref{eq:Z_division} because polynomial structure is stronger than continuity.

% (Householder transformations and an additional sign-variable may be used to form a rational map $\theta \rightarrow Z$, but this is terribly ugly).

Given that a polynomial (continuous) parameterization of $Z$ in terms of $\theta$ alone does not exist in $n=3$, we instead parameterize $Z$ in terms of $\theta \in S^2$ and a new variable $b \in S^2$
using the existence of the three-dimensional cross-product $\times$ operator. For any $b \in S^2$ with $\theta \cdot b = 0$, a  three-dimensional orthogonal coordinate frame can be created as
\begin{align}
\begin{bmatrix}
    \theta & Z
\end{bmatrix} = \begin{bmatrix}
    \theta & b & \theta \times b 
\end{bmatrix}. \label{eq:frame_3}
\end{align}

% However, the dimension-3 case contains a cross-product  $\times$ which can be used to define coordinate frames.
% Letting $b \in S^2$ be a direction, consider the following set:
The frame \eqref{eq:frame_3} can be used to construct the following set:
\begin{align}
        \Psi_L^3 &= \{(\theta, t, b, y) \in  \Omega\times S^2\times  B_R^{2}  \mid (\theta t + b y_1 + (\theta \times b) y_2) \in L, \ \theta \cdot b = 0\}. \label{eq:supp_psi_3}
\end{align}

The set $\Psi_L^3$ from \eqref{eq:supp_psi_3} has 9 variables $(\theta, t, y, b)$, which is fewer than the $3^2 + 3 = 12$ variables needed to represent $\Psi_L$ from \eqref{eq:supp_psi} in terms of $(\theta, t, Z, y)$.  The symmetric and algebraic structures from Section \ref{sec:structure_dim_indep} are retained in the 3-dimensional case. We note that a possible Gr\"{o}bner basis of $\{\norm{\theta}_2^2 = 1, \norm{b}_2^2 = 1, \theta \cdot b =0\}$ in the graded  lexicographical ordering is:
\begin{align}
 &\norm{\theta}_2^2 - 1,   & & \theta_1 (b_2^2 + b_3^2 - 1) - b_1(\theta_2 b_2 - \theta_3 b_1), \nonumber \\
 & \norm{b}_2^2 - 1, & & \theta_1(\theta_2 b_2 + \theta_3 b_3) + b_1(-\theta_2^2 - \theta_3^2 + 1), \label{eq:grobner_3}\\
 &\theta \cdot b, & & \theta_1(\theta_2 b_3^2 -\theta_3 b_2 b_3 - \theta_2) + b_1(b_2 - \theta_2 \theta_3 b_3 + \theta_3 b_2), \nonumber\\
 & & & \theta_2(\theta_2 b_3^2 - 2 \theta_3 b_2 b_3 - \theta_2) + \theta_3^2(b_2^2 - 1) -b_2^2 - b_3^2 + 1. \nonumber
\end{align}

% \urg{TODO: find a SAGBI basis with the $O(2)$ symmetry for $\Psi_L^3$. Find primary and secondary invariants, figure out how to reduce runtime using a Hironaka decomposition \cite{gatermann2004symmetry}.}

The set description $\Psi_L^3$ possesses an $SO(2)$ symmetry, which is a subgroup of the full $O(2)$ symmetry from Section \ref{sec:symmetry_continuous}. 

\begin{lem}
    The three-dimensional set $\Psi_L^3$ is $SO(2) \times \Z_2$-invariant. 
\end{lem}
\begin{proof}
The $\Z_2$ symmetry is described in Section \eqref{sec:symmetry_discrete} with respect to the mapping $(\theta, t, y, b) \leftrightarrow (-\theta, -t, -y, -b)$. We now focus on the $SO(2)$ symmetry.
    Let $Q$ be an orthogonal matrix to transform in the local $y$-plane as \begin{equation}
        Q = \begin{bmatrix}
    Q_1 & Q_2 \\
    Q_3 & Q_4
\end{bmatrix} \in O(2),
    \end{equation}
and let $(\theta, t, b, y)$ be an arbitrary point in $\Psi_L^3$. Define the quantities $(Z', b', y')$ as
\begin{align}
    Z' &= \begin{bmatrix}
        b& \theta \times b
    \end{bmatrix} Q^T & b' &= Z' \begin{bmatrix}
        1 \\ 0
    \end{bmatrix} & \theta \times b' &=  Z'\begin{bmatrix}
        0 \\ 1
    \end{bmatrix}& y' &= Q y.
\end{align}

It holds that
\begin{subequations}
\begin{align}
\norm{b'}_2^2 &= \norm{(Q_1 b + Q_3 (\theta \times b))}_2^2 = Q_1^2 \norm{b}_2^2 + Q_3^2 \norm{\theta \times b}_2^2 = 1 \\
\theta \cdot b' &= Q_1(\theta \cdot b) + Q_3\theta \cdot (\theta \times b) = 0 \\
(\theta \times b') \cdot (Z' [0; 1]) &= Q_1 Q_2 b\cdot (\theta \times b) + Q_1 Q_4 \norm{\theta \times b}_2^2 - Q_2 Q_3 \norm{b}_2^2 - Q_3 Q_4 b \cdot (\theta \times b) \\
&= Q_1 Q_4 - Q_3 Q_2 = \textrm{det}(Q),
\end{align}
\end{subequations}
which implies that $(\theta \times b') = Z'[0; 1]$ when $\textrm{det}(Q)=1$. The implication $(\theta, t, b, y) \in  \Psi^3_L \implies (\theta, t, b', y') \in  \Psi^3_L$ will be valid when $\textrm{det}(Q) = 1$, which is valid only in the subgroup $SO(2) \subset O(2)$.

% it holds that $(\theta, t, b', y') \in \Psi_L^3$. The set $\Psi_L^3$ is therefore invariant under $O(2)$-action.

% \urg{TODO: Does this need more proof text? This argument holds for rotations $\abs{Q}=1$. What about reflections with $\abs{Q}=-1$?}
    
\end{proof}

The $SO(2)$ restriction arises from the imposition of the unique cross-product vector $\theta \times b$ when forming $Z$. Just like in Section \ref{sec:dim_248}, the loss of symmetry (from $O(2)$ to $SO(2)$) is outweighed by the decrease in computational complexity (12 variables down to 9 variables).

% \urg{Write the groebner basis?}

% \section{Reduction of Complexity}

% The dual of the slice-volume program is,
% \begin{subequations}
% \label{eq:meas_slice}
%     \begin{align}
%         M^* = &\sup  \inp{1}{\mu} \label{eq:meas_slice_prob} \\
%         & \inp{1}{\mu_0} = 1 \label{eq:meas_slice_prob}\\
%         & \mu + \hat{\mu} = \mu_0 \label{eq:meas_slice_vol}\\
%         & \mu_0 \in \Mp{S^n \times [0, R]} \label{eq:meas_slice_pt}\\
%         & \mu \in \Mp{\Omega} \label{eq:meas_slice_mu} \\
%         & \hat{\mu} \in \Mp{X \times S^n \times [0, R]}. \label{eq:meas_slice_slack}    
%     \end{align}
% \end{subequations}

% The measure $\mu_0$ in \eqref{eq:meas_slice_pt} is a probability measure \eqref{eq:meas_slice_vol} that selects optimal planes $(\theta, t)$ to maximize the slice volume. The mass $\inp{1}{\mu} \leq \inp{1}{\mu_0 \otimes \lambda_X} = 1 \textrm{Vol}_{n-1}X}$ in the objective is the slice volume.

% \urg{This is confusing}

% \urg{Write this}

\section{Stokes Constraints}
\label{sec:stokes}

\ac{SOS}-truncations to the slice-volume program (in \eqref{eq:slice_upper_sos}) are vulnerable to slow convergence due to the presence of Gibbs phenomena when $w$ is upper-bounding an indicator function in \eqref{eq:slice_upper_sos_psi}. This section will utlize the Stokes constraint method of \cite{tacchi2023stokes} in order to improve the numerical convergence of \ac{SOS} approximations while keeping the \ac{LP} objective the same.

% This section will apply the Stokes constraint method of \

\subsection{Stokes Constraints Background}

We will begin by briefly reviewing Stokes' Theorem and the Stokes constraint formulation for volume approximation from \cite{lasserre2017computing, tacchi2023stokes}. Section \ref{sec:stokes_symmetry} presents new work in incorporating symmetries into Stokes constraints. Refer to sections 2 and 3 of \cite{lasserre2020boundary} for further information about Stokes constraints for \ac{BSA} sets.

\subsubsection{Stokes' Theorems}

The classical Stokes' theorem is defined with respect to a smooth manifold. We will restrict our presentation to a bounded setting.

\begin{thm}
\label{thm:stokes_std}
    Let $M \subset \R^n$ be a compact smooth manifold ($\partial M$ is $C^1$). Then for all  continuous $(n-1)$-forms $\omega$, it holds that
    \begin{align}
        \int_{\partial M} \omega = \int_{M} d \omega. \label{eq:stokes}
    \end{align}
\end{thm}

Stokes' theorem in \eqref{thm:stokes_std} is applicable in the \ac{BSA} setting in which $M$ is defined as $M = \{x \mid g_0(x) \geq 0\}$ for a $g_0 \in \R[x]$ such $\partial M = \{x \mid g_0(x) = 0\}$, and that the  $g_0(x)=0$ level set is singularity-free.

Stokes' theorem may be generalized to sets with nonsmooth structure, such as with corners and singular boundaries. One extension to Stokes' theorem uses the concept of Standard Domains \cite{whitney2012geometric}:
\begin{defn}
\label{defn:standard_domain}
    Let $S \subset \R^n$ be a bounded, open, connected set. Let $\Delta = \textrm{cl}(S) - S$ be the  boundary of $S$. 
    $S$ is a \textbf{standard domain} if it obeys the following properties:
    \begin{enumerate}         
        \item There exists a set $Q \subset \Delta$ in which $Q$ has zero measure w.r.t. the Hausdorff  measure of $\sigma_{\textrm{cl}(S)}$.
        \item For each point $p \in \Delta-Q$, there exists a coordinate assignment $p \mapsto [\theta(p); Z(p)] \in O(n)$, local coordinates $\alpha$ with $x = [\theta(p), \; Z(p)] \alpha$, a neighborhood $N_p \ni p$,
        and a smooth function $h \in C^1(N_p)$.
        \item Points in $\Delta \cap N_p$ can be represented as $\alpha_1 = h(\alpha_2, \ldots, \alpha_n)$.
        \item Points in $S \cap N_p$ can be represented as $\alpha_1 > h(\alpha_2, \ldots, \alpha_n)$.   
    \end{enumerate}
\end{defn}

\begin{rem}
    The interiors of full-dimensional polytopes are examples of standard domains.
\end{rem}

\begin{thm}[Theorem 14.A of \cite{whitney2012geometric}]
    Let $S$ be a standard domain, and let $Q$  and $\Delta$ be chosen in accordance with Definition \ref{defn:standard_domain}.    
    Further assume that $\omega$ is summable over $\Delta,$ $d\omega$ is summable over $S$, $\omega$ is continuous and bounded in $\text{cl}(S) - Q$, and $\omega$ is smooth in $S$. Then
        \begin{align}
        \int_{\text{cl}(S) - S} \omega = \int_{S} d \omega. \label{eq:stokes_std}
    \end{align}
\end{thm}

Refer to \cite{harrison1999flux} for other formulations of Stokes' theorems towards nonsmooth boundaries (including for fractals).

\subsubsection{Stokes Constraints for Volume Approximation}

This subsection will summarize the presentation for Lemma 4.1 of \cite{tacchi2023stokes} and Equation 4.16 of \cite{tacchi2021thesis} with a focus on \ac{BSA} sets described by multiple polynomial inequality constraints.

Consider the $n$-dimensional volume approximation problem from of \eqref{eq:vol_approx} and \eqref{eq:vol_meas} with respect to \iac{BSA} set $L = \{x \mid \forall i: g_i(x) \geq 0\}$. Let $\lambda_L$ and $\sigma_L$ be the Lebesgue and Hausdorff measures of $L$, such that $(\lambda_L, \lambda_X - \lambda_L)$ is the optimal solution of \eqref{eq:vol_meas}. Define $L_i = \{x \mid g_i(x) = 0, \ g_{-i}(x) \geq 0\}$ as the $i$-th `face' of $\partial L$. 
For every function $\zeta \in C^1(\partial L),$ it holds from the definition of $\partial L$ that 
\begin{align}
    \sum_{i=1}^L \left(\int_{L_i} g_i(x) \zeta(x) d \sigma_L\right) = 0,
\end{align}
because the face $L_i$ is defined by $g_i(x) = 0$. 

We will require new assumptions throughout this work:
\begin{itemize}
    \item[A3] The set $\textrm{int}(L)$ is a standard domain.
    \item[A4]The set $L$ satisfies for each $i$ that $\forall x \in L: \nabla g_i(x) \neq 0$.  
\end{itemize}

A consequence of A4 is that the normal-vector map pointing out of $L_i$ is well-defined with 
% Assuming for each $i$ that $\forall x \in L_i: \nabla_x g_i(x) \neq 0$, then the normal vector pointing out of $L_i$ is
\begin{equation}
    \label{eq:normal_map}
    \vec{\mathbf{n}}_i(x) = \nabla g_i(x) / \norm{\nabla g_i(x)}.
\end{equation}

\begin{rem}
    The standard domain description of $L$ is valid with respect to $\textrm{int}(L_i) = \{x \mid  g_i(x) = 0, \ g_{-i}(x) > 0\}$. Given that $g_i(x)=0$ on the boundary component $L_i$, 
    it holds that $\forall \zeta_L \in C^1(\partial L): \int_{L_i} g_i(x) \zeta(x) d \sigma_L(x) = \int_{\textrm{int}(L_i)} g_i(x) \zeta(x) d \sigma_L(x)$.
\end{rem}

For any function $u \in C^1(L)^n$, the following relation may be computed:
\begin{subequations}
\begin{align}
    \int_L \nabla_x \cdot u(x) dx &= \sum_{i=1}^L \int_{\textrm{int} L_i} {u(x) \cdot \vec{n}_i(x) d \sigma_L(x)} \\
    &= \sum_{i=1}^L \int_{\textrm{int} L_i} {(u(x) \cdot \nabla g_i(x)) / \norm{\nabla_x g_i(x)} d \sigma_L(x)}.
    \intertext{Defining $\nu_i \in \Mp{L_i}$ as a measure with density $1/\norm{\nabla_x g_i(x)}$ with respect to $\sigma_L$ on $L_i$, the above sum can be expressed as }
    \int_L \nabla_x \cdot u(x) dx  &= \sum_{i=1}^L \int_{\textrm{int} L_i} {(u(x) \cdot \nabla_x g_i(x)) d \nu_i}.
\end{align}
\end{subequations}

The unique optimum $(\mu, \hat{\mu}) = (\lambda_L, \lambda_X - \lambda_L)$ to \eqref{eq:vol_meas} must therefore furnish the existence of measures $\tilde{\nu}_i \in \Mp{\textrm{int} ({L_i})}$ such that $\forall u \in C(L)^1$:
\begin{align}
    \inp{\nabla_x \cdot u(x)}{\mu(x)} &= \textstyle \sum_{i=1}^L \inp{u(x) \cdot \nabla_x g_i(x)}{\tilde{\nu}_i(x)}. \label{eq:stokes_meas_basic_interior}
\end{align}

We can extend $\tilde{\nu}_i \in \Mp{\textrm{int} ({L_i})}$ to $\nu_i \in \Mp{{L_i}}$ given that $\partial L_i$ has measure 0 with respect to $\sigma_{L_i}$, and given that A4 ensures that the density $1/\norm{\nabla g_i}$ is bounded over $L_i$.

Equation 4.16 of \cite{tacchi2021thesis} is the augmentation of \eqref{eq:vol_meas} with the Stokes constraints of 
\begin{subequations}
\label{eq:stokes_meas_basic}
\begin{align}
    \inp{\nabla_x \cdot u(x)}{\mu(x)} &= \textstyle \sum_{i=1}^L \inp{u(x) \cdot \nabla_x g_i(x)}{\nu_i(x)} \\
    & \nu_i \in \Mp{L_i}.
\end{align}
\end{subequations}
Theorem 4.2 of \cite{tacchi2023stokes} proves that this Stokes infimal volume approximation program achieves its minimum in the case where $L$ is described by the smooth superlevel set of a single polynomial inequality constraint.

\subsubsection{Stokes Constraints with Symmetry}
\label{sec:stokes_symmetry}

% This section 
Stokes constraints can be introduced to volume approximation methods while respecting a symmetry structure of $L$. 
This development is motivated by the presence of the  discrete and continuous symmetries in the slice-volume program \eqref{eq:slice_upper_sos} from Section \ref{sec:structure_dim_indep}. For simplicity, we will focus on the discrete sign symmetry from Section \eqref{sec:symmetry_discrete} (such as in the $n \in \{2, 4, 8\}$ case), but will later remark on how the continuous symmetry from Section \ref{sec:symmetry_continuous} can be accomodated.

Let $G$ be a discrete group with finite cardinality $\abs{G}$. The action of $\rho \in G$ on $x \in \R^n$ is $x \mapsto \rho x$, and the pullback of a $\rho$-action on a measure $\mu \in \Mp{X}$ is $\forall A \in X: \mu(\rho A)$. The measure $\mu$ is invariant under $G$ if $\forall A \in X: \mu(\rho A) = \mu(A)$, and the set of $G$-invariant measures supported on $X$ is $\Mp{X}^G$. Refer to \cite{riener2013exploiting} for more detail about invariant measures and their application in polynomial optimization.

The group-average of a function $f$ over a group $G$ is 
\begin{equation}
    [f(x)]_G = (1/\abs{G}) \sum_{\rho \in G} f(\rho x). \label{eq:group_average}
\end{equation} 

The returned average $f_G(x)$ is a $G$-invariant function. The averaging procedure $[\cdot]_G$ is also called the Reynolds operator.

The pairing $\inp{\cdot}{\cdot}$ satisfies 
\begin{align}
    \forall f \in C(X), \ \mu \in \Mp{X}^G: & \qquad  \inp{f}{\mu} = \inp{f_G}{\mu}. \label{eq:invariant_measure}
\end{align}

If the set $L$ is $G$-symmetric, then the Lebesgue measure $\lambda_X$ and Hausdorff measure $\sigma_X$ are both  $G$-invariant. The Stokes constraint in \eqref{eq:stokes_meas_basic} can be equivalently expressed in a $G$-invariant manner as 
\begin{align}
    \inp{[\nabla_x \cdot u(x)]_G}{\mu(x)} &= \textstyle \sum_{i=1}^L \inp{[u(x) \cdot \nabla_x g_i(x)]_G}{\nu_i(x)}. \label{eq:stokes_meas_G}
\end{align}

\begin{rem}
In the case where $G$ is an infinite group (with a continuous action) with unique Haar (symmetry-invariant) measure $\nu$, the group average in \eqref{eq:group_average} is taken as the integral $1/\textrm{vol}(G) \int_{\rho \in G} f(\rho x) d \nu(\rho)$. 
\end{rem}

\subsection{Stokes Constraints for Slice-Volume}

We adjoin the slice-volume programs \eqref{eq:slice_upper} and \eqref{eq:slice_upper_meas}  with Stokes-type constraints for volume approximation. 

\subsubsection{Stokes Measure Program}
The Stokes constraint in \eqref{eq:stokes_meas_basic} will be performed with respect to the volume-slice-coordinate $y$.

Define the support set $\Psi_L^i$ on the boundary of $\Psi_L$ as
\begin{align}
    \Psi_L^i &= \Psi_L \mid_{g_i(\theta t + y Z) = 0} & & \forall i\in 1..N_c.
\end{align}
New boundary measures will be defined as in $\nu_i \in \Mp{\Psi_L^i}$ to form the Stokes scheme.

The slice-volume program has a subgroup $G$ generated by the reflection $G: (\theta, t, Z, y) \leftrightarrow -(\theta, t, Z, y)$. The averaging operation for a function $f(\theta, t, Z, y)$ over this subgroup is
\begin{align}
    [f(\theta, t, Z, y)]_G &= (f(\theta, t, Z, y)+f(-\theta, -t, -y, -Z))/2.
\end{align}

The induced Stokes constraint from \eqref{eq:stokes_meas_slice} for the constraint $\forall i: g_i(\theta t + Z y) = 0$ with respect to the test functions $u \in C^1(\Psi_L)^n$ is
\begin{subequations}
\begin{align}
    \inp{[\nabla_y \cdot u(\theta, t, Z, y)]_G}{\mu(\theta, t, Z, y)} &= \textstyle \sum_{i=1}^L \inp{[u(\theta, t, Z, y) \cdot \nabla_y g_i(\theta t +  Zy)]_G}{\nu_i(\theta, t, Z, y)},\label{eq:stokes_meas_slice}
    \intertext{which will be written in shorthand  as}
    [\text{grad}_y]_G \mu &= \textstyle \sum_{i=1}^L[\text{grad}_y ((\theta t +  Zy)_\#g_i)]_G \nu_i.
\end{align}
\end{subequations}

\begin{prop}
\label{prop:grad_zero}
    The gradient $\nabla_y g_i(\theta t +  Zy)$ can be computed by the chain rule:
    \begin{align}
    \nabla_y g_i(\theta t +  Zy) = Z^T \nabla_x g_i(\theta t +  Zy). \label{eq:chain_derivative}
\end{align}
This gradient is zero whenever $\theta$ and $\nabla_y g_i$ are colinear, given that $Z^T \theta = 0$ from \eqref{eq:supp_psi}.
\end{prop}

Program \eqref{eq:slice_upper_meas} with Stokes constraints from \eqref{eq:stokes_meas_slice} has the form of
\begin{subequations}
\label{eq:slice_upper_meas_stokes}
\begin{align}
    m^*_s = &\sup \  \inp{1}{\mu} \label{eq:slice_upper_meas_obj_stokes} \\
    &\mu_0 \otimes \lambda_{R}^{n-1} = \mu + \hat{\mu} \label{eq:volume_complement_stokes}\\
    & \inp{1}{\mu_0} = 1 \label{eq:slice_upper_meas_prob_stokes}\\  
    & [\text{grad}_y]_G \mu = \textstyle \sum_{i=1}^L[\text{grad}_y (\theta t +  Zy)_\#g_i]_G \nu_i \label{eq:slice_upper_meas_grad_stokes} \\
    & \mu_0 \in \Mp{\Omega_Z}, \\
    & \hat{\mu} \in  \Mp{\Psi} \\
    & \mu \in \Mp{\Psi_L} \label{eq:slice_upper_meas_L_stokes} \\
    &\nu \in \Mp{L_i}, \ \forall i \in 1..N_c.\label{eq:slice_upper_meas_nu_stokes} 
\end{align}
\end{subequations}

\begin{thm}
\label{thm:stokes_meas_same}
    Programs \eqref{eq:slice_upper_meas} and \eqref{eq:slice_upper_meas_stokes} have the same optimal value with $m^*_s = m^*$ under assumptions A1-A4.
\end{thm}
\begin{proof}     
    Letting $\theta, t\in \Omega$ be an arbitrary slicing direction and offset, we can choose a $Z'$ such that $[\theta, Z'] \in O(n)$. Defining the set $Y = \{y \in B^{n-1}_R \mid \theta' t' + Z' y \in L\}$ with Lebesgue measure $\lambda_Y$, the proof of Theorem \eqref{thm:slice_meas_upper_bound} offers feasible variates of  $\mu_0' = \delta_{\theta=\theta'} \otimes \delta_{t=t'} \otimes \delta_{Z = Z'}$, $\mu' = \mu_0' \otimes  \lambda_Y$, and $\hat{\mu}' = \mu_0' \otimes \lambda_{B^{n-1}_R - Y}$.   
    
    For each $i \in N_c$, define $Y_i = \{y \in B^{n-1}_R \mid \theta' t' + Z' y \in L_i\}$ and let $\sigma_i(y)$ be the Hausdorff measure of $y$ over $Y_i$. 
    While A4 implies that $\forall x \in L_i: \ \nabla_x g_i(x) \neq 0$, Proposition \ref{prop:grad_zero} allows for the possibility of points $y \in Y_i$ with $\nabla_y g_i(\theta' t' + Z' y) = 0$. It is therefore not possible to select $\xi_i'$ as unique measure with a density of $1/\norm{\nabla_y
    g(\theta' t' + Z' y)}$ over $\sigma_i(y)$ with $\nu'_i =\delta_{t=t'} \otimes \delta_{\theta= \theta'} \otimes \xi_i \otimes \delta_{Z = Z'}$, given that $1/\norm{\nabla_y
    g(\theta' t' + Z' y)}$ may be unbounded.

    For a given $(t', \theta', Z')$, every point $y \in Y_i$ with $\nabla_y g_i(\theta' t' + Z' y) =0$ has a normal vector $\nabla_x g_i(\theta' t' + Z' y) = s \theta'$ for some $s \in (\R \setminus 0)$ (perpendicular to the slicing plane). We categorize points $y \in L_i$ with $\nabla_y g_i(\theta' t' + Z' y) =0$ into three exclusive cases:
    \begin{enumerate}
        \item $\exists \epsilon > 0 : B_\epsilon^{n-1}(y) = B_\epsilon^{n-1}(y) \cap Y_i$. This implies that $y$ is in the interior of $Y_i$.
        \item $\forall \epsilon > 0: \textrm{Vol}_{n-1}(B_\epsilon^{n-1}(y) \cup Y_i) = 0$. The point $y$ is isolated in $Y_i$ (up to a volume 0 set).
        \item $y \in \partial Y_i$. The point $y$ is on the boundary of another constraint $g_{i'} = 0$ with $\nabla_y g_{i'}(\theta' t' + Z' y) \neq 0$.
    \end{enumerate}

    % Let $\tilde{Y}_i \subset Y_i$ be the regions of $Y$ with nonzero $(n-1)$-volume, and let $\partial \tilde{Y}_i$ be the boundary of $\tilde{Y}_i$.

    Any $y$ falling into cases 1, 2, or 3 can safely be discarded from volume computation with respect to the Stokes constraint on face $i$ (for case 3, the point $y$ will be relevant on face $i' \neq i$).

    Define $\Delta_i = \{y \in Y_i: \nabla_y g_i(\theta' t' + Z' y) =0\},$ and $\tilde{\sigma}_i \in \Mp{Y_i}$ as the restriction of the Hausdorff measure $\sigma_i \in \Mp{Y_i}$ to the set $Y_i - \Delta_i$. The measure $\zeta_i' \in \Mp{Y_i - \Delta_i}$ can be picked uniquely as the measure with a density of $1/\norm{\nabla_y
    g(\theta' t' + Z' y)}$ over $\tilde{\sigma_i}(y)$.  The Stokes measures $\nu'$ can now be chosen as $\nu'_i =\delta_{t=t'} \otimes \delta_{\theta= \theta'} \otimes \delta_{Z = Z'} \otimes \zeta_i'$.
    
    The addition of the Stokes constraints in \eqref{eq:slice_upper_meas_grad_stokes} did not affect the choice nor feasibility of $(\mu'_0, \mu', \hat{\mu}')$. As such, the objective values of \eqref{eq:slice_upper_meas} and \eqref{eq:slice_upper_meas_stokes} are the same.    
    % We let $\mu_0' = \delta_{\theta=\theta'} \otimes \delta_{t=t'}$ be a probability distribution representing the slicing choice $(\theta', t')$. Let $Z' \in \R^{(n-1) \times n}$ be a matrix such that $[\theta', \, Z'] \in O(n)$. We can construct a set $Y = \{y \in B^{n-1}_R \mid \theta' t' + Z' y \in L\}$ that has a Lebesgue measure $\lambda_Y \leq \lambda_{B}$. The measures of $\mu' = \delta_{t=t'} \otimes \delta_{\theta= \theta'} \otimes \lambda_Y \otimes \delta_{Z = Z'}$ and $\hat{\mu}' = \delta_{t=t'} \otimes \delta_{\theta= \theta'} \otimes \lambda_{B^{n-1}_R - Y}$ can therefore be defined satisfying the domination relation \eqref{eq:volume_complement}.    
    \end{proof}

    \subsubsection{Stokes Function Program}

A  \ac{LP} continuous function formulation for the Stokes-constrained slice-volume approximation problem is
% The weak dual program of \eqref{eq:slice_upper_meas_stokes} (following the steps of Appendix \ref{app:duality}) is
% The multiple-constraint Stokes method of \cite[Equation (4.16)]{tacchi2021moment} applied to Program \eqref{eq:slice_upper_psi} is
% The slice-volume maximization program with the Stokes constrai
\begin{subequations}
\label{eq:slice_upper_stokes}
\begin{align}
    p^*_s = &\inf_{\gamma \in \R} \quad \gamma \\
    & \gamma \geq  \Lambda_R w(\theta, t, Z) & & \forall (\theta, t, Z) \in \Omega_Z  \\
    & w(\theta, t, Z, y) \geq 1 + [\nabla_y \cdot u(\theta, t, Z, y)]_G & & \forall (\theta, t, Z, y) \in \Psi_L \label{eq:slice_upper_psi_stokes}\\
    & -[u(\theta, t, Z, y)\cdot \nabla_y g_i(\theta t + Z y)]_G \geq 0 & & \forall (\theta, t, Z, y) \in \Psi^i_L \\
    & w(\theta, t, Z, y) \geq 0 & &\forall (\theta, t, Z, y) \in \Psi \\
    & w(\theta, t, Z, y) \in C(\Psi) \\
    & u(\theta, t, Z, y) \in [C^{0, 0, 0, 1}(\Psi_L)]^n.
\end{align}
\end{subequations}

\begin{rem}
\label{rem:deg_freedom}
    Constraint \eqref{eq:slice_upper_psi_stokes} is less strict than the requirement in \eqref{eq:slice_upper_psi} that $w \geq 1$ over $\Psi_L$.
\end{rem}

\begin{prop}
\label{prop:stokes_valid}
The optimal values of \eqref{eq:slice_upper} and \eqref{eq:slice_upper_stokes} satisfy $p^*_s \leq p^*$.
\end{prop}
\begin{proof}
    Let $(\gamma, w)$ be a feasible solution to the constraints of \eqref{eq:slice_upper}. The choice of $u=0$ provides a feasible point $(\gamma, w, 0)$ to the constraints of \eqref{eq:slice_upper_stokes}, ensuring the lower-bound of $p^*_s \leq p^*$.
\end{proof}

\begin{thm}
    Problems \eqref{eq:slice_upper_meas_stokes} and \eqref{eq:slice_upper_stokes} are strong duals with $p^*_s = m^*_s$ under A1-A4.
\end{thm}
\begin{proof}
See Appendix \ref{app:duality_stokes}.
    % \urg{What do I write here? Do I repeat Appendix \ref{app:duality} for the Stokes case? Lemma \ref{lem:stokes_bounded} is required first.    
    
    % The work in \cite{tacchi2022stokes} and \cite{tacchi2021thesis} took a different approach. Instead of proving boundedness of $\nu$, an association was drawn between a Poisson PDE and the volume approximation problem. The optimal solution of the functional LP (analogue of \eqref{eq:slice_upper_stokes}) is the solution to this PDE. This shorcutted the need for a strong duality proof by boundedness of measures. Instead, strong-duality is accomplished by showing that the dual lower-bound is optimal.}
\end{proof}

\subsection{Sum-of-Squares Slice-Volume with Stokes Constraints}

\begin{prob}
The degree-$k$ \ac{SOS} truncation of program \eqref{eq:slice_upper_stokes} is 
\begin{subequations}
\label{eq:slice_upper_stokes_sos}
\begin{align}
    p^*_{s, k} = &\inf_{\gamma \in \R} \quad \gamma \\
    & \gamma - \textstyle \Lambda_R w(\theta, t, Z) \in \Sigma[\Omega_Z]_{\leq 2k}\label{eq:slice_upper_stokes_sos_int}\\
    & w(\theta, t, Z, y) - 1 - [\nabla_y \cdot u(\theta, t, Z, y)]_G \in \Sigma[\Psi_L]_{\leq 2k} \label{eq:slice_upper_stokes_sos_psi}\\
    & -[u(\theta, t, Z, y)\cdot \nabla_y g_i(\theta t + Z y)]_G \geq 0 \in \Sigma[\Psi_{L}^i]_{\leq 2k} \label{eq:slice_upper_stokes_sos_ug}\\
    & w(\theta, t, Z, y) \in \Sigma[\Psi]_{\leq 2k} \label{eq:slice_upper_sos_stokes_psi_comp} \\
    & w(\theta, t, Z, y) \in \R[\theta, t, y]_{\leq 2k}\label{eq:slice_upper_stokes_sos_w} \\
    & u(\theta, t, Z, y) \in \R[\theta, t,  Z, y]_{\leq 2k - 2 \max \deg g_i}^n. \label{eq:slice_upper_stokes_sos_u}
\end{align}
\end{subequations}    
\end{prob}

\begin{thm}
\ac{SOS} truncations in\eqref{eq:slice_upper_stokes_sos} will converge as $\lim_{k \rightarrow \infty} p^*_{s,k} \rightarrow p^*$ under A1-A4.
    % \urg{I will write a Stone-Weierstrass proof of convergence of $\lim_{k \rightarrow \infty} p^*_{s, k} \rightarrow p^*$, similar to Appendix \ref{app:poly_approx}.}
\end{thm}
\begin{proof}
    Let $(\gamma_k, w_k)$ be a feasible point to the constraints of \eqref{eq:slice_upper_sos}. Similar to the proof of Proposition \ref{prop:stokes_valid}, the point $(\gamma_k, w_k, 0)$ is feasible for the constraints of \eqref{eq:slice_upper_stokes_sos}. Convergence of $\lim_{k \rightarrow \infty} p^*_{s, k} \rightarrow p^*$ 
    is therefore guaranteed by Theorem \ref{thm:sos_indicator_slice}.
\end{proof}

\begin{rem}
    Although the optimal values of \eqref{eq:slice_upper_sos} and \eqref{eq:slice_upper_stokes_sos} will tend towards $p^*$ in increasing degree, there may not be a correspondence between the finite-degree bounds of $p^*_k$ and $p^*_{s, k}$. In practice, the Stokes bounds of \eqref{eq:slice_upper_stokes_sos} decrease much faster than the Gibbs-vulnerable indicator bounds of \eqref{eq:slice_upper_sos}, due to the degree of freedom in $w$ mentioned in Remark \ref{rem:deg_freedom}.
\end{rem}

 \begin{rem}
     The degree of $u$ in \eqref{eq:slice_upper_stokes_sos_u} is restricted to $2k - 2 \max \deg g_i$ in order to ensure that the term $u \nabla_y g_i$ in constraint \eqref{eq:slice_upper_stokes_sos_ug} has degree at most $2k$. This degree doubling occurs  because $\forall i: \deg g_i(\theta t + Z y) = 2 \deg g_i(x)$. The variable-reduction method in dimension 3 from Section \ref{sec:dim3} results in a tripling of degree, because the substitution in \eqref{eq:supp_psi_3} results in cubic terms from $Z y = [b, \theta \times b] y$. When only translation is considered (with fixed $\theta$ and $Z$), the constraint degree stays the same ($\forall i: \deg g_i(\theta t + Z y) = \deg g_i(x)$), and $u$ can be chosen to be a polynomial of degree $2k - \max \deg g_i$.
 \end{rem}
\section{Numerical Examples}

\label{sec:examples}

% \urg{github link to code}
% MATLAB (2021a) code to generate the below examples is publicly available at
% \url{https://github.com/daishuyu/noise-in-observations}. 

Julia code to generate all experiments is available at \url{https://github.com/Jarmill/slice_volume}. \acp{SDP} were synthesized using Correlative-Term Sparsity (CS-TSSOS) \cite{wang2022cs}, modeled using \texttt{JuMP} \cite{Lubin2023jump}, and solved using Mosek 10.1 \cite{mosek92}.

\subsection{Simple Rectangle}

We begin with a centered 2-dimensional rectangle $X = [-0.5, 0.5] \times [-0.7 \times 0.7]$ inside the ball with $R = 1$, represented by $\{x_1 \geq -0.5, x_1 \leq 0.5, x_2 \geq -0.7, x_2 \leq 0.7\}$.  The maximal slice-plane through $X$ is the line between $(-0.5, -0.7)$ and $(0.5, 0.7)$, which has length (slice-volume) $\sqrt{0.5^2 + 0.7^2} = 2\sqrt{74/100}  \approx 1.7205.$ This maximal slice-volume occurs with $\theta = \pm (0.7, -0.5)/ \sqrt{74/100} \approx \pm (0.8137, -0.5812)$ and $t=0$. Table \ref{tab:centered_rectangle} solves the slice-volume problems \eqref{eq:slice_upper_sos} and \eqref{eq:slice_upper_stokes} for the centered rectangle. All results in this section use the explicit $Z$ substitution in \eqref{eq:Z_division2}. 

\begin{table}[h]
    \centering
    \caption{Centered rectangle  without translation ($P^* = 1.7205$)}
    \begin{tabular}{r|c c c c c c }
         Order & 1 & 2& 3& 4 & 5  & 6   \\
         Indicator \eqref{eq:slice_upper_sos}&  2.0 & 2.0 & 2.0 & 1.9980  &1.9309 & 1.9034 \\
         Stokes \eqref{eq:slice_upper_stokes}&  2.0 & 1.7700 & 1.7241 & 1.7209 & 1.7205 & 1.7205
    \end{tabular}    
    \label{tab:centered_rectangle}
\end{table}

We now consider a perturbation of $X$ by an offset of $(0.1, 0)$, resulting in the set $X_{\text{off}} = [-0.4, 0.6] \times [-0.7, 0.7]$. The maximal slice-volume without translation $(t=0)$ occurs at $\theta = \pm (-0.7, 0.4)/\sqrt{13/20} \approx \pm (-0.8682, 0.4961,)$ with a length of $P^*_{\text{off}} = \sqrt{13/5} \approx  1.6125$. Slice-bounds for the offset rectangle are listed in Table \ref{tab:rect_offset}.

\begin{table}[h]
    \centering
    \caption{Offset rectangle  without translation ($P^* =  1.6125$)}
    \begin{tabular}{r|c c c c c c }
         Order & 1 & 2& 3& 4 & 5  & 6   \\
         Indicator \eqref{eq:slice_upper_sos}&   2.0 & 2.0 & 2.0  &  1.9435 &  1.8905 &  1.8443\\
         Stokes \eqref{eq:slice_upper_stokes}&  2.0  & 1.7669   & 1.6298 &   1.6130 &  1.6126 & 1.6125 
    \end{tabular}    
    \label{tab:rect_offset}
\end{table}

Table \ref{tab:rect_offset_translate} contains bounds for the slice-volume problem with translation $t \in [-1, 1]$. The optimum slice-volume of $P^* = 1.7205$ will occur with $\theta = \pm (0.7, -0.5)/ \sqrt{74/100} \approx \pm (0.8137, -0.5812)$ and $t=-7/(10\sqrt{74}) \approx  0.0814$.

\begin{table}[h]
    \centering
    \caption{Offset rectangle  with translation  ($P^* = 1.7205$)}
    \begin{tabular}{r|c c c c c c }
         Order & 1 & 2& 3& 4 & 5  & 6   \\
         Indicator \eqref{eq:slice_upper_sos}&   2.0 & 2.0 & 2.0  &  1.9654 &  1.9430 &  1.8928\\
         Stokes \eqref{eq:slice_upper_stokes}&  2.0  & 2.0 & 1.9184 &   1.7422 &  1.7201 & 1.7205 
    \end{tabular}    
    \label{tab:rect_offset_translate}
\end{table}

Figure \ref{fig:rectangle_translate_6} plots the function $q(\theta, t) = \int_{B^{n-1}_1} w(\theta, t, y) dy$ found by solving the Stokes system \eqref{eq:slice_upper_stokes} at degree $k=6$ ($\deg q = 12$).

\begin{figure}[!h]
    \centering
    \includegraphics[width=0.8\linewidth]{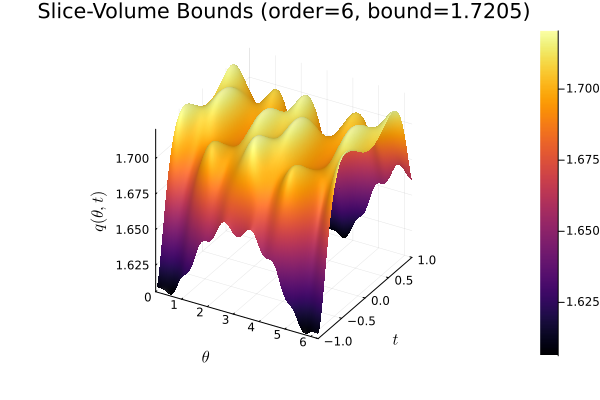}
    \caption{Slicing auxiliary function for offset rectangle in Table \ref{tab:rect_offset_translate}}
    \label{fig:rectangle_translate_6}
\end{figure}

\subsection{Double-Lobe example}

We consider a 2-dimensional Double-Lobe \ac{BSA} set $X$ formed by
\begin{align}
    f(x_1, x_2) &= (x_1)^4 + (x_2)^4 - 3 x_1^2 - x_1 x_2^2 - x_2 + 1 \\
    X &= \{x \in \R^2 \mid f(2.25x_1, 2.25 x_2) \leq 0 \}. \label{eq:double_lobe}
\end{align}

The scaling factor of 2.25 in \eqref{eq:double_lobe} ensures that $X \subset B_{R=1}^2$. The set $X$ is the interior of the nonconvex blue inner regions in Figure \ref{fig:double_lobe}, while the black circle is the boundary of $B_{R=1}^2$.

\begin{figure}[!h]
    \centering
    \includegraphics[width=0.4\linewidth]{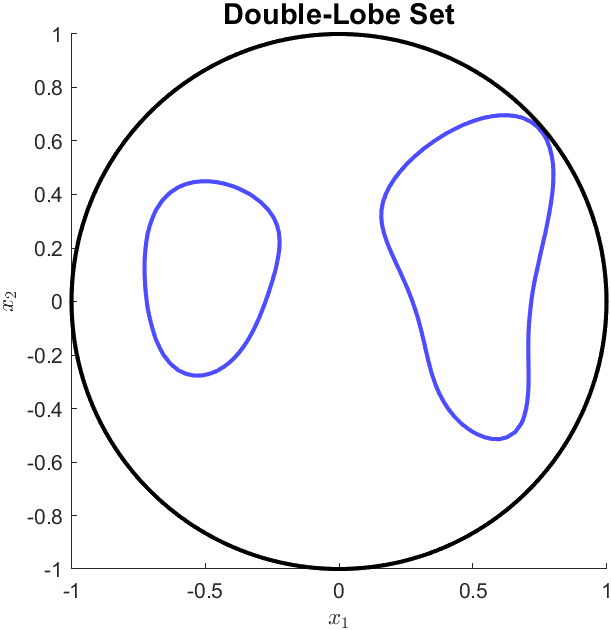}
    \caption{Double-lobe from \eqref{eq:double_lobe}}
    \label{fig:double_lobe}
\end{figure}

\begin{table}[h]
    \centering
    \caption{Double-Lobe without translation}
    \begin{tabular}{r|c c c c c c }
         Order & 1 & 2& 3& 4 & 5  & 6   \\
         Indicator \eqref{eq:slice_upper_sos} &  2.0  &2.0 & 1.9118 &  1.9044 &  1.8270 & 1.5324  \\
         Stokes \eqref{eq:slice_upper_stokes} &  2.0 & 2.0 & 1.6322 & 1.5649 & 1.5256  & 1.0959   
    \end{tabular}    
    \label{tab:double_lobe}
\end{table}

\begin{table}[!h]
    \centering
    \caption{Double-Lobe with translation}
    \begin{tabular}{r|c c c c c c }
         Order & 1 & 2& 3& 4 & 5  & 6   \\
         Indicator \eqref{eq:slice_upper_sos}&  2.0 & 2.0  &1.9910 &   1.9833 &   1.8608 &   1.8294  \\
         Stokes \eqref{eq:slice_upper_stokes}&   2.0 &  2.0 & 1.9799  & 1.6123 &   1.5409 & 1.4728  
    \end{tabular}    
    \label{tab:double_lobe_trans}
\end{table}

\subsection{Double Ellipsoid Cut}

% \urg{TODO: re-run this with the corrected $Z$-dependence $w(t, \theta, y, Z)$. Also, try the $O(2)$ symmetry reduction to get better runtimes.}

This experiment involves a three-dimensional region formed by the exclusion of two ellipsoids. This example is performed with $R = 1$, such that $L$ has the constraint-definition of:
\begin{align}
    L = \begin{Bmatrix}
        x \in \R^3 \mid \begin{array}{l}
             1 - \norm{x}_2^2 \geq 0  \\
             x_1^2 + 4 x_2^2 + 16 x_3^2 - 1 \geq  0 \\
             3.25 x_1^2 -  1.6700 x_1 x_2 - 1.9902 x_1 x_3 + 3.0703 x_2^2 - 2.2158 x_2 x_3 + 2.6796 x_3^2 - 1 \geq 0 
        \end{array} \label{eq:double_ellipse_cut}
    \end{Bmatrix}.
\end{align}

The geometry of the region $L$ is shown in Figure \ref{fig:double_ellipse_cut}. The set $L$ is inside the gray sphere and outside the two ellipsoids.

\begin{figure}
    \centering
    \includegraphics{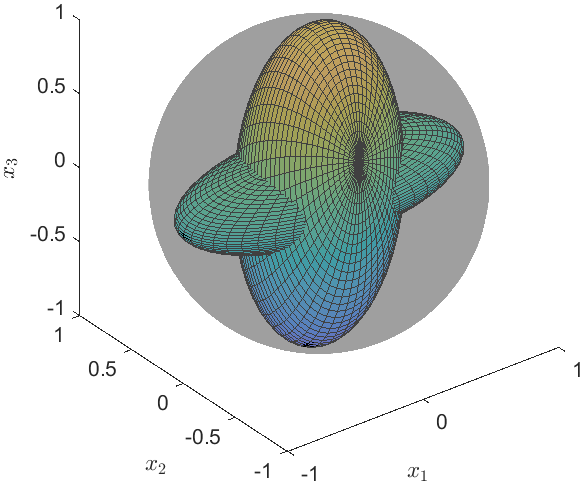}
    \caption{Visualization of double-ellipse-cut region in \eqref{eq:double_ellipse_cut}}
    \label{fig:double_ellipse_cut}
\end{figure}

Table \ref{tab:double_cut_trans} reports on slice-volume bounds for the double-ellipse-cut region in \eqref{eq:double_ellipse_cut} in which only translation in the direction $\theta = [1; 0; 0]$ is allowed (with $[\theta, Z]$ equal to the Identity matrix). The polynomials $(w, u)$ depend on the 3 variables $(t, y_1, y_2)$.

\begin{table}[!h]
    \centering
    \caption{Double-Ellipse-Cut with only translation}
    \begin{tabular}{r|c c c c c c }
         Order & 4 & 5 & 6 & 7& 8& 9   \\
         Indicator \eqref{eq:slice_upper_sos} & 3.1029 & 2.9132 & 2.8556 & 2.7739 & 2.6986 & 2.6653   \\
         Stokes \eqref{eq:slice_upper_stokes}&  2.8013 & 2.5739 & 2.2680 & 2.1333 & 2.1261 & 2.0814  
    \end{tabular}    
    \label{tab:double_cut_trans}
\end{table}

Upper-bounds of the maximal slice-volume under pure rotation for the double-ellipse-cut set at orders $k\in \{1, 2, 3, 4\}$ have values of $\pi$ (with and without Stokes constraints). Computational limitations on the experimental platform prevent the generation and solutions of \acp{SDP} \eqref{eq:slice_upper_sos} and \eqref{eq:slice_upper_stokes} at order $k=4$.

% Table \ref{tab:double_cut_rot} returns slice-volume bounds in the case of pure rotation without translation.
% \begin{table}[!h]
%     \centering
%     \caption{Double-Ellipse-Cut with only rotation}
%     \begin{tabular}{r|c c c c c c }
%          Order & \urg{numbers}  \\
%          Indicator \eqref{eq:slice_upper_sos} & \urg{numbers}  \\
%          Stokes \eqref{eq:slice_upper_stokes}&  \urg{numbers}  
%     \end{tabular}    
%     \label{tab:double_cut_rot}
% \end{table}

% \urg{It is taking a very long time to perform the 3d optimization on my Surface computer. I will run experiments on this next week when I return to my apartment and have my Laptop computer.}
\section{Extensions}
\label{sec:extensions}

This section will detail extensions of the developed slice-volume program to other problems.

\subsection{Radon Transform Maximization}

Consider a (lower-semicontinuous) function $f(x)$ that is compactly supported over the ball $B_R^n$. The Radon-Transform supremization problem for $f$ is 
    \begin{prob}
\label{prob:radon}
Find a direction $\theta$ and an affine offset $t$ to maximize \begin{align}
    F^* =& \sup_{\theta, t} \mathcal{R} f(\theta, t) \label{eq:radon_max}\\
    & \theta \in S^n, \ t \in [-R, R]. \nonumber 
\end{align}
\end{prob}

By using the support-set definition for $\Omega$ from \eqref{eq:supp_omega} and the method from Problem \eqref{eq:slice_cont_std}, we can formulate an \ac{LP} to upper-bound \eqref{eq:radon_max} as
\begin{subequations}
\label{eq:radon_upper}
\begin{align}
    f^* = &\inf_{\gamma \in \R} \quad \gamma \\
    & \gamma \geq  \Lambda_R w(\theta, t, Z) & & \forall (\theta, t) \in \Omega_Z  \label{eq:radon_upper_int}\\
    & w(\theta, t, Z, y) \geq f(\theta t + Z y) & & \forall (\theta, t, y, Z) \in \Psi \label{eq:radon_upper_psi}\\    
    & w(\theta, t, Z, y) \in C(\Psi) \label{eq:radon_upper_w}.
\end{align}
\end{subequations}

Note that constraint \eqref{eq:radon_upper_psi} is equivalent to the pair of \eqref{eq:slice_upper_psi} and \eqref{eq:slice_upper_psi_comp} when $f = I_L$.

\subsection{Radon Transform Approximation}

The program in \eqref{eq:slice_cont_std} aims to find the supremal slice volume of a set $L \subseteq B_R^n$. The produced auxiliary function $w$ produces an upper-bound $\int_{B_R^n} w(\theta, t, y) dy$ on the Radon transform $\mathcal{R} I_L$ (Theorem \ref{thm:slice_meas_upper_bound}), but this tightness only needs to be satisfied at the critical $(\theta^*, t^*)$. 
We can also seek to develop a function that is as close to $\mathcal{R} I_L$ as possible in an $L_1$ sense, thus producing an outer-approximation for the Radon Transform.

% For a given set $L \subseteq B_R^n$, the following program will produce 
% The following program should produce a maximally tight approximation (in an $L_1/W^1$ sense) of this Radon transform parameterized by $(\theta, t)$:
\begin{subequations}
\label{eq:radon_slice_upper}
\begin{align}
    J^* = &\inf \int_{O(n)} \int_{[-R, R]} \int_{B_{R}^{n-1} } w_R(\theta, t, Z, y) dy \, dt \, (d\theta \, dZ) \label{eq:radon_slice_upper_obj}\\
    & w_R(\theta, t, Z, y) \geq 1 & & \forall (\theta, t, y, Z) \in \Psi_L \label{eq:slice_radon_upper_psi}\\
    & w_R(\theta, t, Z, y) \geq 0 & &\forall (\theta, t, Z, y) \in \Psi \\
    & w_R(\theta, t, Z, y) \in C(\Psi).
\end{align}
\end{subequations}

The objective in \eqref{eq:radon_slice_upper} integrates over all variables $(\theta, t, y, Z)$ to develop a maximally-tight approximation.

Problem \eqref{eq:radon_slice_upper} is closely linked to the study of intersection bodies. \cite{gardner1994intersection, Berlow_2022}.
The intersection body $K$ of a star-shaped set $L$ containing the origin is the unique set such that
\begin{align}
    K &= \{y \in \R^n \mid \norm{y}_2 \leq \text{vol}((y \cdot x = 0)  \cap L)\}. \label{eq:intersection_std}
\end{align}

Equivalently, the radial function of the set $K$ is the slice-volume at normal $y/\norm{y}_2$.
Any degree-$k$ moment-\ac{SOS} truncation of \eqref{eq:radon_slice_upper}  with polynomial $w_R \in \R[\theta, t, y]_{\leq 2k}$ and $q_R = \int_{B_{R}^{n-1} } w_R(\theta, t, y) dy$ will satisfy
\begin{align}
     K \subseteq \{\theta t \mid (\theta, t) \in S^n \times [-R,R] : t \leq q(\theta, 0)\},
\end{align} 
thus outer-approximating the intersection body.
% The func
% The function $\int_{B_R^n} w_R(\theta, t, y) dy$
\section{Conclusion}

\label{sec:conclusion}

This work presented the slice-volume problem for compact semialgebraic sets. The slice-volume program was posed the infinite-dimensional primal-dual pair of \acp{LP} \eqref{eq:slice_upper} and \eqref{eq:slice_upper_meas}. These \acp{LP} have no relaxation gap to the original slice-volume task when a bound on $R$ is known (A1). The \acp{LP} can be approximated by the moment-\ac{SOS} hierarchy of \ac{SDP}, with convergence as the degree $k$ tends to infinity. The complexity of these large-scale \acp{SDP} were reduced by applying symmetry, algebraic structure, and topological properties (assignment of $Z$). Stokes constraints for volume computation were incorporated into the slice-volume framework, and the improved convergence of stokes constraints was empirically demonstrated in experiments.

Future work includes determining the convergence rate of the slice-volume scheme, providing conditions under which the infimal Stokes program \eqref{eq:slice_upper_stokes} achieves a minimum, applying SAGBI bases \cite{robbiano2006subalgebra, stillman1999using} to the continuous symmetry reductions in Section \ref{sec:symmetry_continuous},  and finding sets under which slice-volume computation has a simpler structure (such as the polytopes in \cite{brandenburg2023slice}). Another research direction that this framework can be extended towards \ac{SOS}-based analysis of more general integral transformations (such as the X-ray transform \cite{helgason1980radon}). The computational complexity analysis in Sections \ref{sec:complexity} and \ref{sec:reduce_complex} consider an \ac{SDP} framework for polynomial optimization. Further complexity reduction could be obtained with the development of more efficient \ac{SDP} solvers, or by utilizing non-\ac{SDP} optimization strategies  \cite{papp2019sum, cristancho2022harmonic} to solve the slice-volume problem. Other investigations include formulating and implementing these non-\ac{SOS} algorithms for slice-volume optimization, such as by modifying the real-algebraic algorithm from \cite{lairez2019computing}.

% \urg{Conclude the paper. Summarize the problem and important findings. Add some extensions and future work.}

% An extended Arxiv version of this paper is available at \urg{[Arxiv link goes here] what will the arxiv version have that this work does not?}.

\section*{Acknowledgements}

The authors would like to thank Jes\'{u}s de Loera, organizers and participants of the Discrete Optimization semester at ICERM (Jan.-May 2023), Mohab Safey el-Din,  Jie Wang, Mario Sznaier, Roy S. Smith, and the POP group of LAAS-CNRS.

% \urg{Optional acknowledgements.}
% The authors thank Milan Korda for his discussions about occupation measures and time-varying uncertainty.

\bibliographystyle{IEEEtran}
\bibliography{references.bib}

\appendix
\section{Strong Duality of Indicator Slice-Volume}

\label{app:duality}

This proof of strong duality between \eqref{eq:slice_upper} and \eqref{eq:slice_upper_meas} will use notation conventions from Theorem 2.6 of  \cite{tacchi2021thesis}.

\subsection{Weak Duality}

The variables of \eqref{eq:slice_upper} and \eqref{eq:slice_upper_meas} may be expressed as 
\begin{align}
  \bell &= (w, \gamma) \\
  \bbmu &= (\mu_0, \mu, \hat{\mu}). \label{eq:bell}
\end{align}

The residing space of $\bell$ is
\begin{align}
    \mathcal{Y}' &= C(\Psi) \times \R \\
    \mathcal{Y} &= \mathcal{M}(\Psi) \times \R.
\end{align}

The spaces for $\bbmu$ are
\begin{align}
    \mathcal{X}' &=  C(\Omega_Z) \times C(\Psi_L) \times  C(\Psi) \label{eq:dual_spaces}\\
    \mathcal{X} &= \mathcal{M}(\Omega_Z) \times \mathcal{M}(\Psi_L) \times \mathcal{M}(\Psi),\nonumber
\end{align}
with nonnegative subcones of
\begin{align}
    \mathcal{X}'_+ &=  C_+(\Omega_Z) \times C_+(\Psi_L) \times  C_+(\Psi) \label{eq:dual_cones}\\
    \mathcal{X}_+ &= \mathcal{M}_+(\Omega_Z) \times \mathcal{M}_+(\Psi_L) \times \mathcal{M}_+(\Psi).\nonumber
\end{align}

We will also follow convention with \cite{tacchi2021thesis} and refer to $\mathcal{Y}'_+ = \mathcal{Y}'$ and $\mathcal{Y}_+ = \mathcal{Y}$, because there are conic inequalities present in the affine constraints. The feasible sets of \eqref{eq:slice_upper} and \eqref{eq:slice_upper_meas} are $\bell \in \mathcal{Y}'$ and $\bbmu \in \mathcal{X}_+$ respectively. Additionally, the compactness assumption A1 ensures that $(\mathcal{X}_+, \mathcal{X}_+')$ form a pair of topological dual spaces. The set $\mathcal{Y}'$ is equipped with the sup-norm-bounded topology, and the set $\mathcal{X}$ possesses the weak-* topology.

An affine map $\A$ with adjoint $\A'$ may be defined from the constraints of \eqref{eq:slice_upper_meas} as
\begin{align}
    \mathcal{A}(\bbmu) &= [\mu_0 \otimes \lambda_R^{n-1} - \mu - \hat{\mu}, \ \inp{1}{\mu_0}] \\
    \mathcal{A}'(\bell) &= [\gamma - \textstyle \Lambda_R w(\theta, t, y), w, w].\nonumber
\end{align}

The cost and constraint terms of \eqref{eq:slice_upper_meas} can be expressed as
\begin{subequations}
\label{eq:cost_constraint}
    \begin{align}
        \mathbf{b} &= [0, \ 1, \ 0] \\
        \mathbf{c} &= [0, \  0, \ 1].
    \end{align}
\end{subequations}

Pairing of the data in \eqref{eq:cost_constraint} with the variables in \eqref{eq:bell} yield the objectives of \eqref{eq:slice_upper} and \eqref{eq:slice_upper_meas}:
\begin{subequations}
\begin{align}
    \inp{\boldsymbol{\ell}}{\mathbf{b}} &= \gamma \\
    \inp{\mathbf{c}}{\boldsymbol{\mu}} &= \inp{1}{\mu}.
\end{align}
\end{subequations}

A standard-form expression for \eqref{eq:slice_upper} is
\begin{align}
        p^* = &\inf_{\boldsymbol{\ell} \in \mathcal{Y}'_+} \inp{\boldsymbol{\ell}}{\mathbf{b}}
    & &\A'(\boldsymbol{\ell}) - \mathbf{c} \in \mathcal{X}_+, \label{eq:slice_cont_std} \\
    \intertext{and the standard-form expression for \eqref{eq:slice_upper_meas} is }
        M^* =& \sup_{\boldsymbol{\mu} \in \mathcal{X}_+} \inp{\mathbf{c}}{\boldsymbol{\mu}} & & \mathbf{b} - \A(\boldsymbol{\mu}) \in \mathcal{Y}_+. \label{eq:slice_meas_std}
\end{align}

\subsection{Strong Duality}

Sufficient conditions for the strong duality of   \eqref{eq:slice_cont_std} and \eqref{eq:slice_meas_std} by Theorem 2.6 of \cite{tacchi2021thesis}) are:
\begin{itemize}
    \item[R1] The feasible measure solutions of $\bbmu \in \mathcal{X}_+: \mathcal{A}(\bbmu) \in \mathcal{Y}_+$ are bounded.
    \item[R2] There exists a bounded feasible $\bbmu$.
    \item[R3] The vectors in \eqref{eq:cost_constraint} are continuous, and all functions used to describe $\A$ is also continuous.
\end{itemize}

Boundedness of measures is proven in Lemma \ref{lem:meas_bounded}, verifying R1. The proof of \ref{thm:slice_upper_meas} outlines a procedure to obtain a feasible $\bbmu$ from any $(\theta, t) \in \Omega$, fulfilling R2. Requirement R3 is also satisfied, given that the vectors in \eqref{eq:cost_constraint} are constant, and the mapping $w(t, \theta, Z, y) \mapsto \Lambda_R w(\theta, t, Z)$ is continuous in $(\theta, t) \in \Omega$ for every $w \in C(\Psi)$.

Strong duality between \eqref{eq:slice_upper} and \eqref{eq:slice_upper_meas} is therefore proven.

% \urg{Write this}
\section{Polynomial Approximability}

\label{app:poly_approx}

% A proof of Polynomial Approximability based on the Stone-Weierstrass theorem.

This appendix proves that $w(\theta, t, Z, y)$ in \eqref{eq:slice_upper_w} can be taken to be polynomial.

\subsection{Preliminary Polynomial Approximation Lemmas}

In order to prove \ac{SOS} convergence of \eqref{eq:slice_upper_sos} with $\lim_{k \rightarrow \infty} p^*_k = P^*$, we require the following lemma ensuring that it is possible to approximate with polynomials:
\begin{lem}
\label{lem:poly_approx}
    For any $\epsilon \geq 0$ and set $L$ respecting A1, there exists a polynomial $w^p \in \R[\theta, t, Z, y]$ such that $w^p > I_{\Psi_L}$ and $P^* + \epsilon \geq \sup_{(\theta, t, Z) \in \Omega_Z} \Lambda_R  w^p(\theta, t, Z) > P^*$.
\end{lem}
\begin{proof}
    % See Appendix \ref{app:poly_approx}.

Let $(w, \gamma)$ be any feasible solution to \eqref{eq:slice_upper_int}-\eqref{eq:slice_upper_psi_comp} such that $\gamma \geq P^*$. Letting $\eta > 0$ be a tolerance, it therefore holds that $(w + \eta)$ is strictly feasible for \eqref{eq:slice_upper_psi}-\eqref{eq:slice_upper_psi_comp}. There exists a polynomial $w^p \in \R[\theta, t, Z, y]$ such that $\sup_{(\theta, t, Z, y) \in \Psi_L} \abs{w^p(\theta, t, Z, y) - (w(\theta, t, Z, y) + \eta) } \leq \eta/2$ in the compact $\Psi_L$ (A1) by the Stone-Weierstrass theorem. As such, the polynomial $w^p$ is strictly feasible for \eqref{eq:slice_upper_psi}-\eqref{eq:slice_upper_psi_comp}.The integrals are therefore related by
\begin{subequations}
\begin{align}
   \Lambda_R w^p(\theta, t, Z)  &\leq \Lambda_R (w(\theta, t, y) + \eta + \eta/2)  \\
    &= \Lambda_R( w(\theta, t, Z)) + 3\eta/2 \textrm{Vol}_{B^{n-1}_R} \\
    &\leq \gamma + 3\eta/2 \textrm{Vol}(B^{n-1}_R).
\end{align}
\end{subequations}
It therefore holds that for each $\epsilon > 0$, an $\eta > 0$ 
 can be chosen such that $\epsilon > 3\eta/2 \textrm{Vol}(B^{n-1}_R)$. This approximation can become arbitrarily tight by letting $\epsilon, \eta \rightarrow 0$.
\end{proof}

\begin{lem}
\label{lem:ball_constraint}
    Under assumption A1 and A2, the sets $\Omega,$ $\Omega_Z$, $\Psi$, and $\Psi_L$ all are \ac{BSA} sets with ball constraints.
\end{lem}
\begin{proof}
    The set $\Omega$ satisfies a ball constraint, because the description of $\Omega$ includes the constraints $1-\norm{\theta}_2^2 \geq 0$ and $R^2 - t^2 \geq 0$. Therefore, $R_1 = \sqrt{R^2 + 1}$ ensures that $R_1^2 - \norm{\theta}_2^2 - t^2 \in \Sigma[\Omega]$.

    We now consider the set $\Omega_Z $, which includes the constraint $[\theta, \, Z] \in O(n)$. This orthogonality constraint implies that $\norm{\theta}_2^2 + \sum_{ij} Z_{ij}^2 = n$. The set $\Psi_L$ is therefore ball-constrained with $R_2 = \sqrt{2 R^2 + n}$ under $R_2^2 - \norm{\theta}_2^2 - t^2 - \norm{y}_2^2 - \sum_{ij} Z_{ij}^2 \in \Sigma[\Psi_L]$.    The set $\Psi$ is ball-constrained with $R_3 = \sqrt{2R_2^2 + 1}$ and $R_3^2 - \norm{\theta}_2^2 - t^2 - \norm{y}_2^2 - \sum_{ij} Z_{ij}^2\in \Sigma[\Omega \times B_R^{n-1}]$. Given that $\Psi_L \subseteq \Psi$ (and the \ac{BSA} representation of $\Psi_L$ contains all constraints of $\Psi$), it holds that $\Psi_L$ is also ball-constrained.

\end{proof}

\subsection{Proof of Theorem \ref{thm:sos_indicator_slice}}
Lemma \ref{lem:poly_approx} ensures that constraints \eqref{eq:slice_upper_int}-\eqref{eq:slice_upper_psi_comp} can be fulfilled strictly by polynomials. 
    The Putinar Positivestellensatz \cite{putinar1993compact} implies that every positive polynomial over a ball-constrained set is also \ac{WSOS}.  Ball-constraints for the sets of interest are verified by Lemma \ref{lem:ball_constraint}. As such, increasing the degree $k$ sufficiently high will ensure that a feasible polynomial  $w_k$ will be found for each $\epsilon > 0$. Letting $\epsilon \rightarrow 0$ in the limit proves the validity of this theorem.
\section{Strong Duality of Stokes Slice-Volume}

\label{app:duality_stokes}

This appendix proves strong duality between \eqref{eq:slice_upper_stokes} and \eqref{eq:slice_upper_meas_stokes}. The structure of the weak duality portion will follow the format of Appendix \ref{app:duality}.

% This proof of strong duality between \eqref{eq:slice_upper} and \eqref{eq:slice_upper_meas} will use notation conventions from Theorem 2.6 of  \cite{tacchi2021thesis}.

\subsection{Weak Duality}
\label{app:duality_stokes_weak}
The optimization variables of \eqref{eq:slice_upper_stokes} and \eqref{eq:slice_upper_meas_stokes} are
\begin{align}
  \bell &= (w, \gamma, u) \\
  \bbmu &= (\mu_0, \mu, \hat{\mu}, \{\nu_i\}_{i=1}^{N_c}). \label{eq:bell_stokes}
\end{align}
The space containing $\bell$ is
\begin{align}
    \mathcal{Y}' &= C(\Psi) \times \R \times (C^1(\Psi_L))^n\\
    \mathcal{Y} &= \mathcal{M}(\Psi)' \times \R \times ((C^1(\Psi_L)')^n.
\end{align}
The sets defining $\bbmu$ are
\begin{align}
    \mathcal{X}' &=  C(\Omega_Z ) \times C(\Psi_L) \times  C(\Psi) \times \prod_{i=1}^{N_c} C(\Psi_L^i) \label{eq:dual_spaces_stokes}\\
    \mathcal{X} &= \mathcal{M}(\Omega_Z) \times \mathcal{M}(\Psi_L) \times \mathcal{M}(\Psi) \times \prod_{i=1}^{N_c} \mathcal{M}(\Psi_L^i),\nonumber
\end{align}
and their nonnegative subcones of
\begin{align}
    \mathcal{X}'_+ &=  C_+(\Omega_Z) \times C_+(\Psi_L) \times  C_+(\Psi) \times  \prod_{i=1}^{N_c} C_+(\Psi_L^i) \label{eq:dual_cones_stokes}\\
    \mathcal{X}_+ &= \mathcal{M}_+(\Omega_Z) \times \mathcal{M}_+(\Psi_L) \times \mathcal{M}_+(\Psi)  \prod_{i=1}^{N_c} \mathcal{M}_+(\Psi_L^i).\nonumber
\end{align}

The convention of $\mathcal{Y}'_+ = \mathcal{Y}'$ and $\mathcal{Y}_+ = \mathcal{Y}$ is maintained, resulting in feasible set expressions of $\bell \in \mathcal{Y}'$ and $\bbmu \in \mathcal{X}_+$ for \eqref{eq:slice_upper_stokes} and \eqref{eq:slice_upper_meas_stokes}.
% Additionally, the compactness assumption A1 ensures that $(\mathcal{X}_+, \mathcal{X}_+')$ form a pair of topological dual spaces. The set $\mathcal{Y}'$ is equipped with the sup-norm-bounded topology, and the set $\mathcal{X}$ possesses the weak-* topology.

We define an affine map $\A$ and its adjoint $\A'$ from \eqref{eq:slice_upper_meas_stokes} with
\begin{align}
    \mathcal{A}(\bbmu) &= [\mu_0 \otimes \lambda_R^{n-1} - \mu - \hat{\mu}, \ \inp{1}{\mu_0}, \ [\text{grad}_y]_G \mu - \textstyle \sum_{i=1}^L[\text{grad}_y (\theta t +  Zy)_\#g_i]_G \nu_i] \\
    \mathcal{A}^*(\bell) &= [\gamma - \textstyle \Lambda_R w(\theta, t, Z), \ w - [\nabla_y \cdot u]_G, w, \ -[u \cdot \nabla_y g_i]_G],\nonumber
\end{align}
and also set the cost and constraint vectors of \eqref{eq:slice_upper_meas_stokes} as
\begin{subequations}
\label{eq:cost_constraint_stokes}
    \begin{align}
        \mathbf{b} &= [0, \ 1, \ 0, \0_n] \\
        \mathbf{c} &= [0, \  0, \ 1, \0_{N_c}],
    \end{align}
\end{subequations}
to form the following pairings:
\begin{subequations}
\begin{align}
    \inp{\boldsymbol{\ell}}{\mathbf{b}} &= \gamma \\
    \inp{\mathbf{c}}{\boldsymbol{\mu}} &= \inp{1}{\mu}.
\end{align}
\end{subequations}
Under these definitions, the Stokes measure program in \eqref{eq:slice_upper_meas_stokes} is therefore  
\begin{align}    
        m^*_s =& \sup_{\boldsymbol{\mu} \in \mathcal{X}_+} \inp{\mathbf{c}}{\boldsymbol{\mu}} & & \mathbf{b} - \A(\boldsymbol{\mu}) \in \mathcal{Y}_+. \label{eq:slice_meas_std_stokes} \\
        \intertext{and the function program \eqref{eq:slice_upper_stokes} is}
        p^*_s = &\inf_{\boldsymbol{\ell} \in \mathcal{Y}'_+} \inp{\boldsymbol{\ell}}{\mathbf{b}}
    & &\A'(\boldsymbol{\ell}) - \mathbf{c} \in \mathcal{X}_+. \label{eq:slice_cont_std_stokes}
\end{align}
Weak duality implies that $p^*_s \geq m^*_s$.

\subsection{Strong Duality}

We assemble a set of relations between optimal values under A1-A4:
\begin{table}[h]
    \centering
    \caption{\label{tab:stokes} \ac{LP} optimality relations to prove Stokes strong duality}
    \begin{tabular}{c |l}
        $m^*_s = m^*$ &  Theorem \ref{thm:stokes_meas_same}\\
        $m^* = p^*$ &  Theorem \ref{thm:strong_duality_indicator} \\
        $p^* \geq p^*_s $ & Proposition \ref{prop:stokes_valid} \\
        $p^*_s \geq m^*_s$ & Weak duality (Stokes)
    \end{tabular}
\end{table}

\noindent Table \ref{tab:stokes} results in the sandwich of $p^* \leq p^*_s \leq m^*_s$. Given that $m^*_s = m^* = p^*$, it therefore holds that the strong duality relation of $m^*_s = p^*_s$ is proven.

% \begin{rem}
% The strong duality proof in Theorem \ref{thm:strong_duality_indicator} required a boundedness of measure result in Lemma \ref{lem:meas_bounded}.
% This proof of strong duality for the Stokes setting does not require the boundedness of measure variables $\{\nu_i\}$ to \eqref{eq:slice_upper_meas_stokes}. We note that measures $\nu_i$ constructed using the process in \eqref{thm:stokes_meas_same} will be bounded under A1-A4 given that 

% $\nabla_y g_y(\theta t + Z y) = Z^T \nabla_x g_y(\theta t + Z y)$
% $\int_{L_i} 1/\norm{}$
% \end{rem}

% Theorem \ref{thm:stokes_meas_same} states that $m^*_s = m^*$.

\end{document}